\documentclass[a4paper, 10pt]{article}

\usepackage{graphics,graphicx}
\usepackage{amssymb,amsmath}
\usepackage{fullpage}
\usepackage{enumitem}
\usepackage{xcolor}

\usepackage{boldline}

\usepackage[capitalize]{cleveref}

 \usepackage{thmtools}

\newtheorem{theorem}{Theorem}

\newtheorem{lemma}{Lemma}

\newtheorem{problem}{Problem}

\newtheorem{claim}{}[lemma]

\newtheorem{claim2} {Claim}

\newtheorem{cor}{Corollary}

\usepackage{tikz}

\usepackage{tkz-graph}
\usepackage{subcaption}

\def \no {\noindent}

\def \sm {\setminus}
\def \es {\emptyset}

\newenvironment{proof}[1][]%
{\noindent {\setcounter{equation}{0}\it Proof.
	}{#1}{}}{\hfill$\Box$\vspace{2ex}}

\newcommand*\sq{\mathbin{\vcenter{\hbox{\rule{0.75ex}{1.0ex}}}}}

\begin{document}
\title{Tight Chromatic bounds for ($P_2+P_4$, $K_4-e$)-free graphs}
\author{C.~U.~Angeliya\thanks{Computer Science Unit, Indian Statistical
Institute, Chennai Centre, Chennai 600029, India. This research is   supported by National Board of Higher Mathematics (NBHM), DAE, Government of India (No.  02011/55/2023/R\&D-II/3736).}  \and T.~Karthick\thanks{Corresponding author, Computer Science Unit, Indian Statistical
Institute, Chennai Centre, Chennai 600029, India. Email: karthick@isichennai.res.in. ORCID: 0000-0002-5950-9093.  This research is   partially supported by National Board of Higher Mathematics (NBHM), DAE, Government of India (No. 02011/55/2023 NBHM (R. P)/R\&D II/16733).}\and Shenwei Huang\thanks{Co-corresponding author. School of Mathematical Sciences and LPMC, Nankai University, Tianjin 300071, China. Supported by National Natural Science Foundation of China (12161141006) and Natural Science Foundation of Tianjin (20JCYBJC01190). Email: shenweihuang@nankai.edu.cn} }

\maketitle

\begin{abstract}
 For a graph $G$, $\chi(G)$ and $\omega(G)$ respectively denote the chromatic number and clique number of $G$.   In this paper, we show the following results:
  \begin{itemize}
  \item If $G$ is a  ($P_2+P_4$, $K_4-e$)-free graph with $\omega(G)\geq 3$, then $\chi(G)\leq \max\{6, \omega(G)\}$, and   the bound is tight for each $\omega(G)\notin \{4,5\}$.
  \item If $G$ is a  ($P_2+P_4$, $K_4-e$)-free graph with $\omega(G)= 4$, then $\chi(G)= 4$.
  \end{itemize}
 These results extend   the  chromatic bounds known for the class of ($P_2+P_2$, $K_4-e$)-free graphs and for the class of ($P_2+P_3$, $K_4-e$)-free graphs, improve  the bound of Chen and Zhang [arXiv:2412.14524 [math.CO], 2024]  given for the class of ($P_2+P_4$, $K_4-e$)-free graphs, partially answer  a question of Ju and the third author [Theor. Comp. Sci. 993 (2024) Article No.: 114465] on `near optimal colorable graphs',   and    a question of Schiermeyer (unpublished) on the chromatic bound for ($P_7$, $K_4-e$)-free graphs.
\end{abstract}

\medskip
\no{\bf Keywords}:~Graph classes;   Chromatic number; Clique number; Near-optimal colorable graphs.

\section{Introduction}
All our graphs are finite,  non-null, simple and undirected, and we refer to West \cite{west} for undefined notation and terminology.   For an integer $t\geq 1$, let $P_{t}$ and $K_t$ respectively  denote  the chordless path and the complete graph  on $t$ vertices. For an integer $t\geq 2$, a  $K_{t}-e$ is the graph obtained from $K_{t}$  by removing an edge.    For an integer $t\geq 3$, let $C_{t}$    denote  the chordless cycle on $t$ vertices.
An \emph{odd hole} is the graph $C_{2t+1}$  where $t\geq 2$, and an \emph{odd antihole}  is the complement graph of an odd hole.
We say that a graph $G$ \emph{contains} a graph $H$ if $H$ is an induced subgraph of $G$.
Given a   class of graphs $\cal F$,  we say that  a graph $G$ is {\it $\cal F$-free} if $G$ does not contain a graph of $\cal F$. The
\emph{union} of two vertex-disjoint graphs $G_1$ and $G_2$, denoted by $G_1+G_2$, is the graph
with the vertex-set $V(G_1)\cup V(G_2)$ and the edge-set $E(G_1)\cup E(G_2)$.

Let $G = (V, E)$ be a graph with vertex-set $V(G)$ and edge-set $E(G)$. For a graph $G$ and $v\in V(G)$,  $N(v)$ is the   set of neighbors of  $v$ in $G$, and the \emph{degree} of $v$, $deg(v):=|N(v)|$.   Two  nonadjacent vertices $u$ and $v$ in a graph $G$ are said to be \emph{comparable}, if  either $N(u) \subseteq N(v)$ or $N(v) \subseteq N(u)$.
We say that a graph $G$ is \textit{good}  if  $G$ has a pair of comparable vertices or there is a vertex with degree at most 3, and we say that $G$ is \emph{bad} otherwise. For a graph $G$ and a vertex-subset $X$ of $G$, the graph $G[X]$ is subgraph induced by $X$ in $G$, and  $G-X$ is the induced subgraph $G[V(G)\sm X]$.

A {\it $k$-coloring} of a graph $G$ is a function $f: V(G)\rightarrow \{1,2,\ldots, k\}$ such that $f(u)\neq f(v)$ whenever $u$ and $v$ are adjacent vertices in $G$.
As usual, for a graph $G$, $\chi(G)$ and $\omega(G)$ respectively stand for the chromatic number and the clique number of $G$. Obviously   for every graph $G$, we have $\chi(G)\geq \omega(G)$. A graph $G$ is \emph{perfect} if  each of its induced subgraph $H$ satisfies $\chi(H)=\omega(H)$. It is well-known that the class of perfect graphs include  the class of bipartite graphs, the class of complement graphs of bipartite graphs and the class of $P_4$-free graphs, etc.
A celebrated result of Chudnovsky et al. \cite{spgt} gives a necessary and sufficient condition for the class of perfect graphs via forbidden induced subgraphs, and is now known as the `Strong perfect graph theorem  (SPGT)'. It states that \emph{a graph is perfect if and only if it does not contain an odd hole or an odd antihole}.

The problem of coloring  the class of ($P_r+P_t$)-free graphs, where $r,t\geq 2$ and its subclasses have been studied in a variety of contexts recently; see \cite{BC,AK3,CLZ,KMMNPS,LLW,SR-Poly-Survey,PFF}. Here  we are interested in the problem of obtaining tight chromatic bounds for the class of ($P_r+P_t$, $K_4-e$)-free graphs, and \cref{tab} gives the current status of the problem for ($H$, $K_4-e$)-free graphs for some $H\cong P_r+P_t$. Since the class of ($P_r+P_t$, $K_4-e$)-free graphs is a subclass of the class of ($P_{r+t+1}$, $K_4-e$)-free graphs, it is known from a  general result of Schiermeyer and Randerath \cite{SR-Poly-Survey} that  every  ($P_r+P_t$, $K_4-e$)-free graph $G$ with $\omega(G)=k\geq 2$ satisfies $\chi(G)\leq (k-1)(r+t-1)$.

 \begin{table}[t]
 \begin{center}
 {\footnotesize{
\begin{tabular}{|p{1cm}V{3.0}p{1.48cm}p{0.7cm}|p{1.45cm}p{0.28cm}|p{1.45cm}p{0.28cm}|p{1.65cm}p{0.2cm}|cp{0.2cm}|}
\hline
 ~~~~$H$   &  $\omega(G) =2$ && $\omega(G) =3$& & $\omega(G) =4$ & & $\omega(G) =5$ &&~~$\omega(G) =k\geq 6$&\\
\hlineB{3.0}
 $P_2+P_2$ & $\chi(G) \leq 3$ $^\star$  &\cite{BRSV} & $\chi(G)= 3$ $^\star$  &\cite{BRSV} & $\chi(G) = 4$ $^\star$ & \cite{BRSV}  & $\chi(G) = 5$ $^\star$ & \cite{BRSV} & $\chi(G) = k$ $^\star$  &\cite{BRSV} \\
\hline
$P_2+P_3$ & $\chi(G) \leq 4$ $^\star$  &\cite{RS-Survey}  & $\chi(G)\leq 6$ $^\star$  &\cite{KSM-GC} & $\chi(G) = 4$ $^\star$  &\cite{PFF}  & $\chi(G) = 5$ $^\star$ & \cite{BC}& $\chi(G)=k$ $^\star$  &\cite{BC}\\
\hline
$P_2+P_4$ & $\chi(G) \leq 4$ $^\star$  &\cite{BC,  BGPS}  & $\chi(G)\leq 6$ $^\star$  & *  & $\chi(G) = 4$ $^{\star}$ & *  & $\chi(G) \leq 6$ $^?$ & *   &   $\chi(G)=k$  $^\star$    & * \\

\hline
$P_3+P_3$ & $\chi(G) \leq 4$ $^\star$  &\cite{Pya}  & ~~~~~? & & ~~~~~?&& ~~~~~?&  &  ? & \\
\hline

\end{tabular}

\caption{Chromatic bounds for some  ($H$, $K_4-e$)-free graphs $G$ with given clique number, where $H\cong P_r+P_t$. Here  `$*$' indicates the contributions of this paper, `$\star$' indicates that the given bound is tight, and `?' indicates that the problem of finding the tight chromatic bound is open. }\label{tab}
}}
\end{center}
\end{table}

 A class of graphs $\cal C$ is  \emph{near optimal colorable} \cite{Ju-Huang} if there exists a constant $c\in \mathbb{N}$ such that $\chi(G) \leq \max\{c,~\omega(G)\}$ for each graph $G\in \cal C$. Obviously  the class of near optimal colorable graphs  includes the class of perfect graphs  and it is an interesting  subclass of the class of $\chi$-bounded graphs which are well explored in the literature; see  \cite{AK-Survey, SR-Poly-Survey}.
Among other intriguing   questions,  Ju and the third author \cite{Ju-Huang} posed  the following meta problem.

\begin{problem} [\cite{Ju-Huang}] \label{JH-Problem}
Given a  class of $(F, K_{\ell}-e)$-free graphs $\cal G$, where  $F$ is  a forest which is not an induced subgraph of a $P_4$ and $\ell\geq 4$, does there exist a  constant $c\in \mathbb{N}$ such that  $\cal G$ is near optimal colorable?  If so, what is the smallest possible constant $c$?
\end{problem}

It is known that   the class of ($P_5, K_5-e$)-free graphs \cite{AK}, and the class of ($P_6$, $K_4-e$)-free graphs are near optimal colorable \cite{GHJM}. Many classes of $(F, K_{\ell}-e)$-free graphs for various $F$ are now known to be near optimal colorable; see \cite{AAK, CKLV,  Ju-Huang-2, SR-Poly-Survey}  and the references therein.
Here we   focus on \cref{JH-Problem} when $F=P_{2}+P_{t}$ where $t\geq 2$. While the class of ($P_2+P_2$, $K_4-e$)-free graphs and the class of ($P_2+P_3$, $K_4-e$)-free graphs are near optimal colorable (see \cref{tab}), surprisingly  the status of \cref{JH-Problem} is unknown even when $F=P_2+P_4$ and $\ell=4$. First we note the following.

%

\begin{theorem}[\cite{BC}]\label{w2} If $G$ is a ($P_2+P_4$, $K_4-e$)-free graph with $\omega(G)= 2$, then $\chi(G)\leq 4$  and   the bound is tight.
\end{theorem}

Very recently, Chen and Zhang \cite{CW} showed that if $G$ is a ($P_2+P_4$, $K_4-e$)-free graph with $\omega(G)= 3$, then $\chi(G)\leq 7$
and if $\omega(G) = 4$, then $\chi(G)\leq 9$, and that $\chi(G)\leq 2\omega(G) - 1$ when  $\omega(G) \geq 5$.
Here we improve the above results of Chen and Zhang \cite{CW} and prove the following results.

\newpage
\begin{theorem}\label{main-c5c7}
 Let $G$ be a ($P_2+P_4$, $K_4 -e$)-free graph with $\omega(G)\geq 3$. Then one of the following holds.
  \begin{enumerate}[label=($\roman*$)]\itemsep=0pt
 \item $G$ is perfect.
 \item If $G$ contains a  $C_5$, then either $\chi(G)\leq 6$ or $\chi(G)= \omega(G)$.
 \item If $G$ is $C_5$-free  and contains a  $C_7$, then   $\chi(G)\leq 3$.
 \end{enumerate}
 \end{theorem}

\begin{theorem}\label{mainthm}
If $G$ is a  ($P_2+P_4$, $K_4-e$)-free graph with $\omega(G)\geq 3$, then $\chi(G)\leq \max\{6, \omega(G)\}$. Moreover the bound is tight for each $\omega(G)\notin \{4,5\}$.
\end{theorem}

\noindent{\bf Proof of \cref{mainthm} assuming \cref{main-c5c7}}.~Since an odd hole $C_{2t+1}$  where $t\geq 4$  contains a $P_2+P_4$, and since an odd antihole $\overline{C_{2p+1}}$  where $p\geq 3$  contains a $K_4-e$, the proof of the first assertion of \cref{mainthm} follows as a direct consequence of SPGT \cite{spgt} and  from   \cref{main-c5c7}. The proof of the second assertion of \cref{mainthm} is  immediate  since the complement graph of the 16-regular Schl\"afli graph on 27 vertices   is a ($P_2+P_4$, $K_4-e$)-free graph $G$ with $\omega(G)= 3$ and $\chi(G)=6$, and since the complete graph on at least 6 vertices  is   a ($P_2+P_4$, $K_4-e$)-free graph $G$ with $\chi(G)=\omega(G)\geq 6$. \hfill{$\Box$}

\medskip

 Moreover, Chen and Zhang \cite{CW} showed that if $G$ is ($P_2+P_4, K_4-e, C_5$)-free with $\omega(G)\geq 5$, then $G$ is perfect.  Next as an immediate  consequence  of \cref{main-c5c7}, we show that
 the lower bound on $\omega$ can be improved.

 \begin{cor}\label{cor-per}
 If $G$ is a ($P_2+P_4$, $K_4 -e$, $C_5$)-free graph with $\omega(G)\geq 4$, then $G$ is perfect.
 \end{cor}
\noindent{\bf Proof of \cref{cor-per} assuming \cref{main-c5c7}}.~Let $G$ be a ($P_2+P_4$, $K_4 -e$, $C_5$)-free graph with $\omega(G)\geq 4$. Suppose to the contrary that $G$ is not perfect.
 Then by SPGT, $G$ contains a $C_7$, and hence from \cref{main-c5c7}:$(iii)$, we have $\omega(G)\leq \chi(G)\leq 3$ which is a contradiction.
 This proves \cref{cor-per}. \hfill{$\Box$}

\medskip
The assumption that the clique number $\omega\geq 4$ in \cref{cor-per} cannot be dropped and that the constant 4 cannnot be lowered. For instance, if  $G^*$ is the graph that consists of a $C_7$  with the vertex-set $\{v_1,v_2,\ldots, v_7\}$ and the edge-set $\{v_1v_2,v_2v_3,\ldots, v_6v_7,v_7v_1\}$ plus a vertex $x$ which is adjacent to $\{v_1,v_2, v_4, v_6\}$ and is nonadjacent to $\{v_3,v_5,v_7\}$, then $G^*$ is a ($P_2+P_4$, $K_4 -e$, $C_5$)-free graph with $\omega(G^*)=3$ which is not perfect.  Moreover, the assumption that the graph is $C_5$-free in \cref{cor-per} cannot be dropped. For instance, if  $G'$ is the graph that consists of the graph $K_{t}$ ($t\geq 4$) which shares an edge with the $C_5$ so that $|V(G')|=t+3$, then $G'$ is a ($P_2+P_4$, $K_4 -e$)-free graph that contains a $C_5$ with $\omega(G')=t$ and is not perfect. However, we prove the following theorem when $\omega=4$.

\begin{theorem}\label{w4}
Let $G$ be a ($P_2+P_4$, $K_4 -e$)-free graph with $\omega(G) = 4$. Then $\chi(G) =4$.
\end{theorem}
\begin{proof}Let $G$ be a ($P_2+P_4$, $K_4 -e$)-free graph with $\omega(G) = 4$. We may assume that $G$ is connected. If $G$ is a good graph,  then we prove the theorem
by using induction on $|V(G)|$.  If $G$ has a pair of comparable vertices, say $u$ and
	$v$, such that $N(u) \subseteq N(v)$ (say), then  we can take any $4$-coloring of $G-\{u\}$ and extend it to a $4$-coloring of $G$, using the color of $v$ for the vertex $u$.
If $G$ has a vertex, say $v$, such that $deg(v) \leq 3$, then   we can take any $4$-coloring of $G -\{v\}$ and extend it to a $4$-coloring of $G$, using for $v$ a color (possibly new) that does not appear in $N(v)$, and we are done. So we may assume that $G$ is bad, and  then the theorem follows from \cref{mainthm-2} given below.
\end{proof}

\begin{theorem}\label{mainthm-2}
 Let $G$ be a bad  ($P_2+P_4$, $K_4-e$)-free graph with $\omega(G) = 4$. Then $\chi(G) = 4$.
\end{theorem}


Furthermore, since the  class of ($P_7, K_4-e$)-free graphs includes the class of ($P_2+P_4$, $K_4-e$)-free graphs, we observe that  \cref{w2,mainthm,w4} altogether partially answers the following problem of Ingo Schiermeyer (unpublished).
\begin{problem}
Is it true that, every ($P_7, K_4-e$)-free graph $G$ satisfies
$\chi(G) \leq \omega(G) + \varepsilon$, where $\varepsilon\geq 3$  is a positive integer?
\end{problem}

In \cref{genprop-1}, we give some general properties of   ($P_2+P_4, K_4-e$)-free graphs  that  contain a  $C_5$, and the rest of the paper is devoted to the proofs of \cref{main-c5c7,mainthm-2} (and they are given in \cref{sec-main,sec:w4} respectively). We finish this section with some more preliminaries which are used in this paper. For a graph $G$, a vertex $v\in V(G)$  is \emph{complete} (resp. \emph{anticomplete}) to $X\subseteq V(G)\setminus \{v\}$ if  $v$ is adjacent (resp. nonadjacent) to every vertex in $X$.   
For any two vertex-subsets
$S$ and $T$ of $G$, we denote by $[S,T]$, the set of edges which has
one end in $S$ and the other   in $T$.  An edge-set   $[S,T]$ is \emph{complete} (or $S$ is \emph{complete} to $T$)   if each vertex in $S$
is complete to $T$. An edge-set   $[S,T]$ is \emph{anticomplete} (or $S$ is \emph{anticomplete} to $T$)   if each vertex in   $S$
is  anticomplete to $T$.
For a positive integer $k$, we write $\langle k \rangle$ to denote the set $\{1, 2, \ldots , k\}$, and we write $i\in \langle k \rangle$ if $i\in \{1, 2, \ldots , k\}$ and $i$ mod $k$.

\medskip
We will use the following well-known results often.

\vspace{-0.25cm}

\begin{enumerate}[label=  ($\mathbb{R}$\arabic*), leftmargin=1.05cm] \itemsep=0pt
\item\label{p3-free} {\it If a graph is $P_3$-free, then it can be expressed as a disjoint union of complete graphs.}
\item\label{p4-free} {\it Every $P_4$-free graph is perfect.}
\end{enumerate}

We will also use the following  lemma whose proof is obvious.

\begin{lemma}\label{obs1}
Let $G$ be a  ($K_4-e, K_5$)-free graph. Suppose that $G$ contains a $K_3$, say $K$. Then: (i) Any vertex in
$V(G)\setminus V(K)$ is either complete to $K$ or it is adjacent to at most one vertex in $K$. (ii) No two vertices in $V(G)\setminus V(K)$ is complete to $K$.
\end{lemma}

 \section{Structural properties of  ($P_2+P_4, K_4-e$)-free graphs}\label{genprop-1}
In this section,  we prove some  important and useful  structural properties  of a ($P_2+P_4$,\,$K_4-e$)-free graph  that  contains a $C_5$.

Let $G$ be a ($P_2+P_4, K_4-e$)-free graph.
Suppose that $G$ contains a  $C_5$, say with
 vertex-set $C:= \{v_1, v_2, v_3, v_4, v_5\}$ and edge-set $\{v_1 v_2, v_2 v_3, v_3 v_4, v_4 v_5,v_5 v_1\}$.  Then since $G$ is ($K_4-e$)-free, we immediately see that for any $v\in V(G)\setminus C$, there does not exist an index $i \in \langle  5 \rangle$ such that $\{v_{i}, v_{i+1}, v_{i+2}\} \subseteq N(v)\cap C$. Using this fact, we   define the following sets. For $i \in \langle  5 \rangle$, we let:
\begin{center}
 \begin{tabular}{ccl}
 $A_{i}$ &:=  & $\{ v \in V(G)\setminus C \mid N(v) \cap  C  = \{v_i\}\},$\\
 $B_{i}$ &:=  & $\{v \in V(G)\setminus C \mid  N(v) \cap C = \{v_i, v_{i+1}\}\},$\\
 $D_{i}$ &:=  & $\{ v \in V(G)\setminus C \mid N(v) \cap C = \{v_{i-1}, v_{i+1}\}\},$\\
 $Z_{i}$ &:= & $\{ v \in V(G)\setminus C \mid N(v) \cap C = \{v_{i-2}, v_i, v_{i+2}\}\},$ and\\
  $T $ &:=  & $\{ v \in V(G)\setminus C \mid N(v) \cap C = \es\}$.\\
 \end{tabular}
 \end{center}
 Further we let $A:= \cup_{i=1}^5A_i$,   $B:= \cup_{i=1}^5B_i$, $D:= \cup_{i=1}^5D_i$ and $Z:= \cup_{i=1}^5Z_i$ so that $V(G)= C\cup A\cup B\cup D\cup Z\cup T$. Moreover  for each $i\in \langle  5 \rangle$, the following hold: (We note that the properties \ref{DZ}, \ref{x} and \ref{xto} and their proofs are  explicitly available in \cite{AAK}  and implicitly in several other papers,   we give them here for completeness.)

\begin{enumerate}[label=  $(\mathbb{O}\arabic*)$, leftmargin=1.5cm] \itemsep=0pt
\item \label{DZ} {\it $D_{i}$ is a stable set, and $|Z_{i}|\leq 1$.}
 \item \label{x}  {\it $B_i\cup Z_{i-2}$ is a clique. Moreover, each vertex in $B_i\cup Z_{i-2}$ has at most one neighbor in  $B_{i+2}\cup Z_{i}$, and has at most one neighbor in  $B_{i-2}\cup Z_{i+1}$.}
\item  \label{xto}
 {\it $[B_i\cup Z_{i-2},~ A_{i}\cup A_{i+1}\cup B_{i-1}\cup B_{i+1}\cup (D\sm D_{i-2})\cup (Z\sm Z_{i-2})]$ is an empty set.}

\item \label{adt}    {\it     $A_i\cup T$ is a  stable set.}
\item \label{ab}   {\it  If $A_i\neq \emptyset$, then  $|B_{i+1}| \leq 2$ and $|B_{i-2}|\leq 2$.}
\item \label{aa1}{\it        $[A_i, A_{i+1} \cup A_{i-1}]$ is either complete or an empty set.}
 \item \label{b}  {\it If $B_i\neq \es$, then the following hold:
        \begin{enumerate}[label=  $(\roman*)$]
            \item\label{bb} $|B_{i-1}| \leq 1$ and $|B_{i+1}| \leq 1$.
              \item\label{bzd}Either $B_{i-1}$ is an empty set or $ D_{i-1} \cup D_{i+1}\cup Z_i$ is an empty set.
            \item\label{b2d} Either $|B_{i}|=1$ or $D_i \cup D_{i+1}$ is an empty set.
            \item\label{da} $[ A_i \cup D_{i+1}, D_{i-1}]$ and $[A_i, D_i \cup A_{i+1}]$ are empty sets. Likewise, $[A_{i+1} \cup D_i, D_{i+2}]$ and $[A_{i+1}, D_{i+1}]$ are empty sets.
\item\label{tz} Either $|B_i| \leq 2$ or $[T, Z\setminus Z_{i-2}]$ is an empty set.
         \end{enumerate}
}
  \end{enumerate}

\begin{proof}
\noindent{\ref{DZ}}:~If there are  adjacent vertices  in $D_i$, say $d$ and $d'$, then $\{v_{i-1}, d, d', v_{i+1}\}$ induces a $K_4-e$; so $D_i$ is a stable set.
Next suppose to the contrary that there are vertices, say $z$ and $z'$ in $Z_i$. Now if  $z$ and $z'$ are adjacent, then $\{v_{i}, z, z', v_{i+2}\}$ induces a  $K_4-e$, and if $z$ and $z'$   are nonadjacent, then $\{z, v_{i-2}, v_{i+2},z'\}$ induces a $K_4-e$.  These contradictions show that $|Z_i| \leq 1$. So \ref{DZ} holds.  $\sq$

\medskip
\noindent{\ref{x}:}~If there are nonadjacent vertices in $B_i\cup Z_{i-2}$, say $u$ and $v$ , then $\{u, v_i, v_{i+1}, v\}$ induces a $K_4-e$; so $B_i\cup Z_{i-2}$ is a clique. Now  if there is a vertex in $B_i\cup Z_{i-2}$, say $x$, which has two neighbors in $B_{i+2}\cup Z_{i}$, say $u$ and $v$, then by the first assertion, $\{x,u,v,v_{i+2}\}$ induces a $K_4-e$; so each vertex in $B_i\cup Z_{i-2}$ has at most one neighbor in  $B_{i+2}\cup Z_{i}$.  Likewise, each vertex in $B_i\cup Z_{i-2}$ has at most one neighbor in  $B_{i-2}\cup Z_{i+1}$.   This proves \ref{x}.  $\sq$

\medskip
\noindent{\ref{xto}:}~If  there are   vertices, say $u\in B_i\cup Z_{i-2}$ and $v\in A_{i}\cup A_{i+1}\cup B_{i-1}\cup B_{i+1}\cup (D\sm D_{i-2})\cup (Z\sm Z_{i-2})$ such that $uv\in E(G)$, then \{$v_{i}$, $ v_{i+1}$, $u, v$\} induces a $K_4-e$. So \ref{xto} holds. $\sq$

\medskip
\noindent{\ref{adt}}:~If there are adjacent vertices in $A_i \cup T$, say $u$ and $v$, then $\{u, v, v_{i+1}, v_{i+2}, v_{i+3}, v_{i+4}\}$ induces a $P_2+P_4$. So \ref{adt} holds. $\sq$

\medskip
\noindent{\ref{ab}}:~Let $a\in A_i$. Suppose to the contrary that there are three vertices, say $b, b'$ and $b''$, in $B_{i+1}$. Recall   from \ref{x}  that $\{b,b',b''\}$ is a clique. Then since $\{a,b,b',b'', v_{i+2}\}$ does not induce a $K_4-e$, clearly $a$ is nonadjacent to at least two vertices in $\{b,b',b''\}$, say $b$ and $b'$. But then $\{b, b', a,v_i, v_{i-1}, v_{i-2}\}$ induces a $P_2+P_4$ which is a contradiction. So $|B_{i+1}| \leq 2$. Likewise, $|B_{i-2}|\leq 2$. This proves \ref{ab}. $\sq$

\medskip
\noindent{\ref{aa1}}:~If there are vertices, say $u, u' \in A_i$ and $v\in A_{i+1}$ such that $uv \in E(G)$ and $u'v\not\in E(G)$,  then  $\{v_{i+2}, v_{i-2}, u', v_i, u, v\}$ induces a $P_2+P_4$, by \ref{adt}. So \ref{aa1} holds.  $\sq$

\medskip
\noindent{\ref{b}}:~\ref{bb}: If there are vertices, say $u,v \in B_{i-1}$, then for any $b\in B_i$, from \ref{x} and \ref{xto}, we have $uv \in E(G)$ and $bu, bv \not\in E(G)$,  and then $\{u,v, b,v_{i+1}, v_{i+2}, v_{i-2}\}$ induces a $P_2+P_4$. Hence $|B_{i-1}| \leq 1$. Likewise, $|B_{i+1}| \leq 1$. So \ref{bb} holds. $\diamond$

\medskip
\noindent{\ref{bzd}}:~If  there are vertices, say $u\in B_{i-1}$  and $v \in Z_i\cup D_{i-1}\cup D_{i+1}$, then for any $b\in B_i$, from \ref{xto}, we have $uv,bu,bv \notin E(G)$, and then   $\{b, v_{i+1}, u, v_{i-1}, v_{i-2}, v\}$ or $\{u, v_{i-1}, b, v_{i+1}, v_{i+2},v\}$ induces a $P_2+P_4$. So \ref{bzd} holds. $\diamond$

\medskip
\noindent{\ref{b2d}}:~If  there are vertices, say $b, b'\in B_i$  and $d \in D_i\cup D_{i+1}$, then $\{b, b',   v_3, v_4, v_5, d\}$ induces a $P_2+P_4$, by \ref{x} and \ref{xto}. So \ref{b2d} holds. $\diamond$

\medskip
\noindent{\ref{da}}:~Let $b\in B_i$. If there are adjacent vertices, say $u\in D_{i-1}$ and $v\in A_i \cup D_{i+1}$, then \{$b$, $v_{i+1}$, $v$, $u$, $v_{i-2}$, $v_{i-1}$\} induces a  $P_2+P_4$,  by \ref{xto}. Also, if there are adjacent vertices, say $u\in D_{i} \cup A_{i+1}$ and $v\in A_i$, then \{$v_{i+2}$, $v_{i-2}$, $u$, $v$, $v_i$, $b$\} induces a $P_2+P_4$,   by \ref{xto}. So \ref{da} holds. $\diamond$

\medskip
\noindent{\ref{tz}}:~Suppose to the contrary that there are three vertices in $B_i$, say $u,v$ and $w$, and that there are adjacent vertices, say   $t\in T$ and $z\in Z \sm Z_{i-2}$. Then from \ref{x}, since $\{v_i, u,v,w,t\}$ does not induce a $K_4 - e$, we may assume that $tu,tv\notin E(G)$, and then from \ref{xto}, we see that $\{u,v,t, z, v_{i-1}, v_{i-2}\}$ or $\{u,v,t,z,v_{i-2},v_{i+2}\}$ induces  a $P_2 + P_4$ which is a contradiction.  So \ref{tz} holds.
  \end{proof}

\section{Proof of \cref{main-c5c7}} \label{sec-main}
In this section, we give a proof of \cref{main-c5c7}, and the proof indeed follows from a sequence of lemmas which are based on some special graphs $F_1,F_2,F_3$ and $F_4$ (see Figure~\ref{Gi}), and it is given at the end of this section.

 \begin{figure}[h]
\centering
 \includegraphics[height=2.75cm]{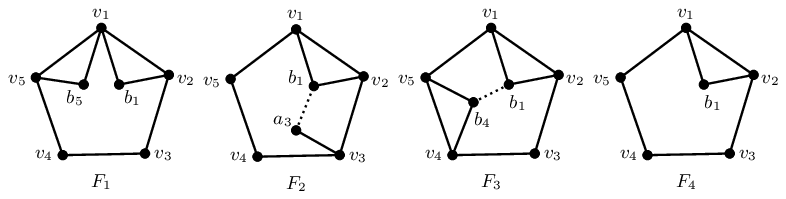}
\caption{Some specific graphs. Here  a dotted line  between two vertices indicates that such vertices may or may not be adjacent.  }\label{Gi}
\end{figure}

\begin{lemma}\label{F1}
    Let $G$ be a ($P_2+P_4, K_4-e$)-free graph. If $G$ contains an $F_1$, then $\chi(G)\leq 6$.
\end{lemma}
\begin{proof}
     Let $G$ be a ($P_2+P_4, K_4-e$)-free graph. Suppose that $G$ contains an $F_1$. We label such an $F_1$ as shown in Figure \ref{Gi}, and we let $C := \{v_1, v_2, v_3, v_4, v_5\}$. Then with respect to $C$, we define the sets $A$, $B$, $D$, $Z$ and $T$ as in  \cref{genprop-1}, and we use the properties in  \cref{genprop-1}. Clearly  $b_1 \in B_1$ and $b_5 \in B_5$. By \ref{xto}, we have $b_1b_5\notin E(G)$. Then from \ref{b}, we have $|B_i| \leq 1$ for each $i\in \{1,2,4,5\}$, $D_2 \cup D_5 \cup Z_1 = \emptyset$, $[A_1, D_1] = \emptyset$, $[A_2, D_3] = \emptyset$ and $[A_5, D_4] = \emptyset$. Also if there are at least three vertices in $B_3$, then from \ref{x}, it follows that there exists a vertex, say $b_3 \in B_3$ such that $b_3 b_1, b_3 b_5 \not\in E(G)$, and then $\{b_1, v_2,b_5, v_5, v_4, b_3\}$ induces a $P_2+P_4$ which is a contradiction; thus $|B_3|\leq 2$.  Now if $B_3\neq \es$, then we let $B_3'$ consist  of one vertex of $B_3$, otherwise we let $B_3'=\es$. Then we define the sets
     $S_1 := \{v_3\} \cup A_1 \cup B_1 \cup B_5\cup D_1 $,
        $S_2 := A_2 \cup B_2 \cup D_3 $,
        $S_3 := \{v_4\} \cup A_3 \cup T$,
        $S_4 := \{v_2, v_5\} \cup A_4 \cup B_3'$,
        $S_5 := \{v_1\} \cup A_5 \cup B_4 \cup D_4 $ and
        $S_6 := (Z\setminus Z_1)\cup (B_3\sm B_3')$ so that $V(G)=S_1\cup S_2\cup \cdots \cup S_6$. Then from \ref{DZ}, \ref{xto} and \ref{adt}, and from above arguments, clearly $S_j$ is a stable set, for   each $j\in \{1,2,\ldots, 6\}$, and hence  $\chi(G) \leq 6$. This proves \cref{F1}.
\end{proof}

\begin{lemma}\label{F2}
    Let $G$ be a ($P_2+P_4$, $K_4 - e$)-free graph. If $G$ contains an $F_2$, then   either $\chi(G)\leq 6$ or $\chi(G)= \omega(G)$.
\end{lemma}
\begin{proof}
 Let $G$ be a ($P_2+P_4$, $K_4 -e$)-free graph. Suppose that $G$ contains an $F_2$. We label such an $F_2$ as shown in Figure \ref{Gi}, and we let $C := \{v_1, v_2, v_3, v_4, v_5\}$. Then with respect to $C$, we define the sets $A$, $B$, $D$, $Z$ and $T$ as in  \cref{genprop-1}, and we use the properties in  \cref{genprop-1}. Clearly   $b_1\in B_1$ and $a_3 \in A_3$. From \cref{F1}, we may assume that $G$ is $F_1$-free; and so $B_2   \cup B_5 = \emptyset$.  From \ref{b}:\ref{da}, we have $[A_1, D_1] = \emptyset$, $[A_1,D_5]=\es$ and $[A_2,D_2]=\es$. Also since $A_3\neq \es$, from \ref{ab},  we have $|B_1|\leq 2$; hence $|B_1\sm \{b_1\}|\leq 1$.  Moreover  we claim the following.

\begin{claim}\label{F2-B4}
 We may assume that $B_4$ is an empty set.
 \end{claim}
\noindent{\it Proof of \cref{F2-B4}}.~Suppose that $B_4\neq \es$, and let $b_4\in B_4$. Then since $G$ is $F_1$-free, it follows from \ref{xto}  that $B_3= \emptyset$. Also  from \ref{b}:\ref{da}, it follows that $[A_j, D_j] = \emptyset$, for $j\in \{4,5\}$, and since $A_3\neq \es$, from \ref{ab},  we have   $|B_4| \leq 2$; hence $|B_4\sm \{b_4\}| \leq 1$.
    Now we define the sets
        $S_1 := \{ v_4, b_1\} \cup A_1 \cup D_1 \cup Z_3$,
        $S_2 := \{v_5\} \cup A_2 \cup (B_1\sm \{b_1\}) \cup D_2 $,
        $S_3 := \{v_2\} \cup A_3 \cup T$,
        $S_4 := \{v_1, b_4\} \cup A_4 \cup D_4$,
        $S_5 := \{v_3\} \cup  (B_4\sm \{b_4\})\cup A_5 \cup D_5$ and
        $S_6 := D_3 \cup Z_1 \cup Z_2 \cup Z_4 \cup Z_5$ so that $V(G)=S_1\cup S_2\cup \cdots \cup S_6$. Then from \ref{DZ}, \ref{xto} and \ref{adt}, and from above arguments, clearly $S_j$ is a stable set, for each  $j\in \{1,2,\ldots, 6\}$, and hence  $\chi(G) \leq 6$.  So we may assume that $B_4= \es$. $\sq$

\smallskip
Next:
\begin{claim}\label{F2-B1=1}
 We may assume that $B_1 =\{b_1\}$.
\end{claim}
\noindent{\it Proof of \cref{F2-B1=1}}.~Suppose that $B_1 \sm \{b_1\} \neq \es$, and let $b_1'\in B_1 \sm \{b_1\}$.  Then from \ref{ab}, we have $B_1=\{b_1, b_1'\}$, and from \ref{b}:\ref{b2d}, we have $D_1 \cup D_2= \emptyset$. Also from \ref{x}, we have $b_1b_1'\in E(G)$. Further since $\{a_3, b_1, b_1', v_1\}$ does not induce a $K_4 -e$ or $\{ b_1, b_1', a_3, v_3, v_4, v_5\}$ does not induce  a $P_2+P_4$, we may assume that  $a_3 b_1 \in E(G)$ and $a_3 b_1' \not\in E(G)$.  Moreover, we have the following.
    \begin{enumerate}[label=($\roman*$)]\itemsep=0pt
     \item\label{F22-A2}  If there is a vertex, say $a_2\in A_2$, then by using \ref{xto}, we see that $\{v_4, v_5, a_3, b_1, v_2, a_2\}$ or $\{v_4, v_5,a_2, $ $ a_3, b_1, b_1'\}$   induces a $P_2 + P_4$; so $A_2 = \emptyset$.

\item\label{F22-B3}  If there are three vertices, say $u,v,w \in B_3$, then by using \ref{x}, we may assume that $ub_1,vb_1\notin E(G)$, and then from \ref{xto}, we see that $\{u, v, a_3, b_1, v_1, v_5\}$   induces a $P_2 + P_4$; so $|B_3|\leq 2$.
\item \label{F22-A3B1}   For any $a' (\neq a_3) \in A_3$,  if $a'b_1'\in E(G)$, then since $\{a', b_1, b_1', v_1\}$ does not induce a $K_4 -e$, we have $a'b_1\notin E(G)$, and then $\{v_4, v_5, a_3, b_1,b_1', a'\}$  induces  a $P_2+P_4$ (by \ref{adt}); so $[A_3, \{b_1'\}] = \emptyset$, since $a'$ is arbitrary.

    \item\label{F22-A1D3}  For any $a\in A_1$ and $d\in D_3$, since $\{ b_1, b_1',a, d, v_4, v_5\}$ does not induce a $P_2 + P_4$ (by \ref{xto}), we have $[A_1, D_3] = \emptyset$.

    \end{enumerate}
Now from \cref{F2-B4}, we have $V(G)=C\cup (A\sm A_2)\cup B_1\cup B_3 \cup D_3\cup D_4\cup D_5\cup Z\cup T$.  If $B_3\neq \es$, then we let $B_3'$ consist of one vertex of $B_3$, otherwise we let $B_3'=\es$. Then we define the sets
        $S_1 :=  \{v_4, b_1'\} \cup A_3$,
        $S_2 := \{v_3, v_5, b_1\} \cup A_1 \cup D_3$,
        $S_3 :=  A_4 \cup B_3' \cup Z_2$,
        $S_4 := \{v_2\} \cup A_5 \cup T$,
        $S_5 := \{v_1\} \cup  (B_3\sm B_3')\cup D_4 \cup Z_5$ and
        $S_6 := D_5 \cup Z_1 \cup Z_3 \cup Z_4$ so that $V(G)= S_1\cup S_2\cup \cdots \cup S_6$. Then from \ref{DZ}, \ref{xto} and \ref{adt}, and from above arguments, we conclude that $S_j$ is a stable set, for each  $j\in \{1,2,\ldots, 6\}$, and hence  $\chi(G) \leq 6$, and we are done. So we   may assume that $B_1 =\{b_1\}$. $\sq$

\smallskip
Next:
\begin{claim}\label{F2-B3}
 We may assume that $B_3$ is an empty set.
 \end{claim}
\noindent{\it Proof of \cref{F2-B3}}.~Suppose that $B_3\neq \es$, and let $b_3\in B_3$. Then from \ref{b}:\ref{da}, we have $[A_3, D_3] = \emptyset = [A_4,D_4]$. Now if $|B_3| \leq 2$, then we define the sets
        $S_1 := \{v_3, b_1\} \cup A_1 \cup D_1$,
        $S_2 :=  \{v_4\} \cup A_2   \cup D_2 \cup Z_5$,
        $S_3 := \{v_5, b_3\} \cup A_3 \cup D_3$,
        $S_4 := \{v_1\} \cup  A_4 \cup  (B_3\setminus\{b_3\}) \cup D_4$,
        $S_5 := \{v_2\} \cup A_5 \cup T$ and
        $S_6 := D_5 \cup Z_1 \cup Z_2 \cup Z_3 \cup Z_4$ so that $V(G)= S_1\cup S_2\cup \cdots \cup S_6$ (by \ref{F2-B4} and \ref{F2-B1=1}), and then from \ref{DZ}, \ref{xto} and \ref{adt}, and from above arguments, we conclude that $S_j$ is a stable set, for each  $j\in \{1,2,\ldots, 6\}$, and hence  $\chi(G) \leq 6$ and we are done. So we may assume that $|B_3| \geq 3$. Then there are vertices, say $u$ and $v$ in $B_3\sm \{b_3\}$. Thus from \ref{x}, since $\{u,v, b_3,v_3,v_4\}$ is a clique, we have $\omega(G)\geq 5$. Then from \ref{ab} and \ref{b}:\ref{b2d}, we have $A_2 \cup A_5 \cup D_3 \cup D_4 = \emptyset$ and from \ref{b}:\ref{tz}, we have $[T, Z\sm Z_1]=\es$. Now  from \ref{F2-B4} and \ref{F2-B1=1}, we have $V(G) = C \cup A_1 \cup A_3 \cup A_4 \cup \{b_1\} \cup B_3 \cup D_1 \cup D_2 \cup D_5 \cup Z \cup T$,
  and we let $W_1:=(B_1\cup Z_4) \cup (B_3\cup Z_1)$ and $W_2:= \{v_5\} \cup A_3 \cup D_2 \cup D_5 \cup Z_5$. Then the following hold.
 \begin{enumerate}[label=($\roman*$)]\itemsep=0pt
     \item\label{F2-W1} Since  $\{v_i, v_{i+1}\}$ is complete to $B_i\cup Z_{i-2}$ for any $i$, from \ref{x}, clearly $G[W_1]$ is perfect with $\omega(G[W_1])\leq \omega(G)-2$; so $\chi(G[W_1])\leq \omega(G)-2$.
     \item\label{F2-W2} From \ref{DZ}, \ref{xto}, \ref{adt}   and \ref{b}:\ref{da}, since $A_3 \cup D_2 \cup Z_5$ and $\{v_5\} \cup D_5$ are stable sets, we see that $G[W_2]$ is a  bipartite graph, and so $\chi(G[W_2])\leq 2$.
    \item\label{F2-W12}  If there are vertices, say $a\in A_3$ and $x \in B_1\cup Z_4$ such that $ax\in E(G)$, then from \ref{x}, we may assume that $xu, xv \notin E(G)$, and then from \ref{x} and \ref{xto}, we see that $\{u,v, a, x, v_1, v_5\}$ induces a $P_2 + P_4$; so $[A_3, B_1\cup Z_4] = \emptyset$. This together with \ref{xto} imply that $[W_1, W_2]=\es$.
     \item\label{F2-A4D1}  If there are vertices, say $a\in A_4$ and $d \in D_1$ such that $ad\in E(G)$, then since $\{v_4,u,v,b_3,d\}$ does not induce a $K_4-e$, we may assume that $du, dv \notin E(G)$, and then from \ref{x} and \ref{xto}, we see that $\{u,v, a, d, v_5, v_1\}$ induces a $P_2 + P_4$; so $[A_4, D_1] = \emptyset$. So from \ref{DZ}, \ref{xto}, \ref{adt}    and  \ref{b}:\ref{tz}, we see that $\{v_1, v_3\} \cup A_4 \cup D_1 \cup Z_2$ and $\{v_2, v_4\} \cup A_1 \cup Z_3 \cup T$ are stable sets.
 \end{enumerate}
  From \ref{F2-W1}, \ref{F2-W2} and \ref{F2-W12}, and since $\omega(G)\geq 5$, we conclude that $\chi(G[W_1 \cup W_2]) = \max\{\chi(G[W_1]), \chi(G[W_2])\} \leq    \max\{\omega(G) - 2, 2\}\leq \omega(G) - 2$. Also from \ref{F2-A4D1}, we see that $\chi(G- (W_1 \cup W_2))\leq 2$. Hence $\chi(G)= \omega(G)$, and we are done. So  we may assume that $B_3= \es$. $\sq$

\smallskip
From \ref{F2-B4}, \ref{F2-B1=1} and \ref{F2-B3}, we conclude that $V(G)= C \cup  A \cup  \{b_1\} \cup D \cup Z \cup T$. Next:

 \begin{claim}\label{F2-D1}
 We may assume that $D_1$ is an empty set.
 \end{claim}
 \noindent{\it Proof of \cref{F2-D1}}.~Suppose that $D_1\neq \es$, and let $d_1\in D_1$.   Then we see that the following hold.
          \begin{enumerate}[label=($\roman*$)]\itemsep=0pt
     \item\label{F2-D1-A3b1} For $a'\in A_3$, since $\{a',v_3,d_1,v_5,v_1,b_1\}$ or $\{b_1,v_1,d_1,a',v_3,v_4\}$ does not induce a $P_2 + P_4$ (by \ref{xto}), we see that $[A_3, \{b_1\}]$ is complete. In particular, $a_3b_1\in E(G)$.
         \item\label{F2-D1-A13} For any $a\in A_1$ and $a'\in A_3$, since $\{v_4,v_5, a,a',b_1,v_2\}$ does not induce a $P_2 + P_4$ (by \ref{xto} and \ref{F2-D1-A3b1}), we have $[A_1, A_3]=\es$.
     \item\label{F2-D1-A1Z1} For any $a_1\in A_1$ and $z_1\in Z_1$, since $\{b_1, v_2,a_1, z_1, v_4, v_5\}$  does not induce a $P_2 + P_4$ (by \ref{xto}), we have $[A_1, Z_1] = \emptyset$.
     \item\label{F2-D1-A12}
         If there are vertices, say $a_1\in A_1$ and $a_2\in A_2$, then from \ref{b}:\ref{da}, we have $a_1a_2\notin E(G)$ and $a_1a_3\notin E(G)$ (by \ref{F2-D1-A13}), and then from \ref{xto} and \ref{F2-D1-A3b1}, one of $\{v_4,v_5,a_2,v_2,b_1,a_3\}$ or  $\{a_1,v_1,a_2,a_3,v_3,v_4\}$  induces a $P_2 + P_4$; so one of $A_1$ or $A_2$ is an empty set.
          \end{enumerate}
Now if $A_2=\es$ or $A_4=\es$, then we define the sets
        $S_1 := \{b_1\} \cup D_2  \cup D_5 \cup Z_3 $,
        $S_2 :=  \{v_2, v_5\} \cup A_1 \cup A_3 \cup Z_1$,
        $S_3 := \{v_1, v_3\} \cup A_5\cup T$,
        $S_4 :=  D_1 \cup D_3 \cup Z_2 \cup Z_4 \cup Z_5$,
        $S_5 :=  \{v_4\} \cup D_4$ and
        $S_6 :=  A_2\cup A_4$   so that $V(G)= S_1\cup S_2\cup \cdots \cup S_6$, and then from \ref{DZ}, \ref{xto}, \ref{adt}, \ref{b}:\ref{da} and from above properties, we conclude that  $S_j$ is a stable set, for each $j\in \{1,2,\ldots, 6\}$, and hence  $\chi(G) \leq 6$.
      So we may  assume that  $A_2\neq \es$ and $A_4\neq \es$, and let $a_2\in A_2$ and $a_4\in A_4$. Then from \ref{F2-D1-A12}, we have $A_1=\es$. Then since
      $\{v_4,v_5,a_2,v_2,b_1,a_3\}$ does not induce a $P_2 + P_4$ (by \ref{F2-D1-A3b1} and \ref{xto}), we  have $a_2a_3\in E(G)$, and then since
      $\{v_1,v_5,v_3,a_3,a_2,a_4\}$ or $\{v_4,a_4,v_1,v_2,a_2,a_3\}$ does not induce a $P_2 + P_4$, we have $a_3a_4\in E(G)$. So
      from \ref{aa1},   $[A_3,A_4]$ is complete. Hence if there are adjacent vertices, say
    $a\in A_3$ and $d\in D_4$, then $\{a, d, a_4, v_3\}$ induces a $K_4 - e$ or $\{v_1, v_2, v_4, a_4, a, d\}$ induces $P_2 + P_4$; so $[A_3,D_4]=\es$. Now  we define the sets
        $S_1 := \{b_1\} \cup D_2 \cup D_5 \cup Z_3 $, $S_2 :=  \{v_2, v_5\} \cup A_4 \cup Z_1$, $S_3 := \{v_1, v_3\} \cup A_5 \cup T$,   $S_4 :=  D_1 \cup D_3 \cup Z_2 \cup Z_4 \cup Z_5$,        $S_5 :=  \{v_4\} \cup A_3\cup D_4$ and  $S_6 :=  A_2$ so that $V(G)= S_1\cup S_2\cup \cdots \cup S_6$. Then from \ref{DZ}, \ref{xto}, \ref{adt} and  \ref{b}:\ref{da}, and from above properties, we conclude that $S_j$ is a stable set, for each  $j\in \{1,2,\ldots, 6\}$, and hence  $\chi(G) \leq 6$.  So we may assume that $D_1 = \emptyset$. $\sq$

   \medskip
By using \cref{F2-B4} to \cref{F2-D1}, we see that  $V(G) = C \cup  A \cup  \{b_1\} \cup(D\sm D_1)\cup Z \cup T$.
 Now if $[A_3, A_4] = \emptyset$,  then we define the sets
        $S_1 :=  \{v_3, v_5, b_1\} \cup A_1 \cup D_5$,
        $S_2 := A_2 \cup D_2 \cup Z_4$,
        $S_3 := \{v_2, v_4\} \cup A_5 \cup T$,
        $S_4 := \{v_1\} \cup A_3 \cup A_4$,
        $S_5 := D_3$ and
        $S_6 := D_4 \cup Z_1 \cup Z_2 \cup Z_3 \cup Z_5$ so that $V(G)= S_1\cup S_2\cup \cdots \cup S_6$, and then from \ref{DZ}, \ref{xto}, \ref{adt} and  \ref{b}:\ref{da}, we conclude that $S_j$ is a stable set, for each $j\in \{1,2,\ldots, 6\}$, and hence  $\chi(G) \leq 6$; so we may assume that $[A_3, A_4] \neq \emptyset$. Then from \ref{aa1}, we see that $[A_3, A_4]$ is complete. Now if there are vertices, say $a_4 \in A_4$ and $d_3 \in D_3$ such that $a_4d_3\in E(G)$, then  $\{v_1, v_5, v_3, a_3, a_4, d_3\}$ induces a $P_2 + P_4$ or $\{a_3, d_3, a_4, v_4\}$ induces $K_4 - e$; so $[A_4,D_3]=\es$. Then we define the sets
            $S_1 :=  \{v_3, v_5, b_1\} \cup A_1 \cup D_5$,
        $S_2 := A_2 \cup D_2 \cup Z_4$,
        $S_3 := \{v_2, v_4\} \cup A_5 \cup T$,
        $S_4 := \{v_1\} \cup A_3$,
        $S_5 :=  A_4\cup D_3$
         and
        $S_6 := D_4 \cup Z_1 \cup Z_2 \cup Z_3 \cup Z_5$ so that $V(G)= S_1\cup S_2\cup \cdots \cup S_6$. Then from \ref{DZ}, \ref{xto} and \ref{adt}, and from above properties, we conclude that $S_j$ is a stable set, for each  $j\in \{1,2,\ldots, 6\}$, and hence  $\chi(G) \leq 6$, and we are done. This proves \cref{F2}.
 \end{proof}

\begin{lemma}\label{F3}
    Let $G$ be a ($P_2+P_4$, $K_4 -e$)-free graph. If $G$ contains an $F_3$, then  either $\chi(G)\leq 6$ or $\chi(G)= \omega(G)$.
\end{lemma}
\begin{proof}
 Let $G$ be a ($P_2+P_4$, $K_4 -e$)-free graph.  Suppose that $G$ contains an $F_3$. We label such an $F_3$ as shown in Figure~\ref{Gi}, and we let $C := \{v_1, v_2, v_3, v_4, v_5\}$. Then with respect to $C$, we define the sets $A$, $B$, $D$, $Z$ and $T$ as in  \cref{genprop-1}, and we use the properties in  \cref{genprop-1}. Clearly    $b_1 \in B_1$ and $b_4 \in B_4$. From \cref{F1,F2}, we may assume that $G$ is $(F_1, F_2)$-free, and hence it follows that $B_2 \cup B_3 \cup B_5 \cup A_1 \cup A_3 \cup A_5 = \emptyset$.
    Now if $|B_1| \leq 2$ and $|B_4| \leq 2$, then we define the sets
        $S_1 := \{v_1, v_4\} \cup A_2 \cup T$,
        $S_2 :=  \{v_2, b_4\} \cup A_4 \cup D_4$,
        $S_3 := \{v_3, v_5, b_1\} \cup D_5$,
        $S_4 := (B_1\sm \{b_1\}) \cup Z_2$,
        $S_5 :=  (B_4\sm \{b_4\}) \cup D_1 \cup D_3$ and
        $S_6 := D_2 \cup Z_1 \cup Z_3 \cup Z_4 \cup Z_5$ so that $V(G)= S_1\cup S_2\cup \cdots \cup S_6$. Then from \ref{DZ}, \ref{xto}, \ref{adt}, \ref{b}:\ref{da}, and from above properties, we conclude that $S_j$ is a stable set, for each $j\in \{1,2,\ldots, 6\}$, and hence  $\chi(G) \leq 6$, and we are done.
  So we may assume that $|B_1|\geq 3$.   Then from \ref{b}:\ref{b2d}, we have $D_1 \cup D_2 = \emptyset$, and so $V(G) = C \cup A_2 \cup A_4 \cup  B_1 \cup B_4 \cup D_3 \cup D_4 \cup D_5 \cup Z \cup T$. Also since $\{v_1,v_2\}\cup B_1$ is a clique (by \ref{x}), we have $\omega(G)\geq 5$. Now we let $W_1:= (B_1\cup Z_4) \cup (B_4\cup Z_2)$ and $W_2:= \{v_3\} \cup D_3 \cup D_5 \cup Z_3$.
   Then the following hold.
 \begin{enumerate}[label=($\roman*$)]\itemsep=0pt
     \item\label{F3-W1} Since $\{v_i, v_{i+1}\}$ is complete to $B_i\cup Z_{i-2}$ for any $i$, from \ref{x}, clearly $G[W_1]$ is perfect with $\omega(G[W_1])\leq \omega(G)-2$; so $\chi(G[W_1])\leq \omega(G)-2$.
     \item\label{F3-W2} From   \ref{DZ} and \ref{xto}, since $D_5 \cup Z_3$ and $\{v_3\} \cup D_3$ are stable sets, we see that $G[W_2]$ is a  bipartite graph, and so $\chi(G[W_2])\leq 2$.
    \item\label{F3-W12}  From \ref{xto}, we have $[W_1, W_2]=\es$.
    \item\label{F3-A2D4} If there are adjacent vertices, say $a\in A_2$ and $d\in D_4$, then since $|B_1|\geq 3$ and since $G[\{d\}\cup B_1]$ does not induce a $K_4-e$, there are vertices, say $u$ and $v$, in $B_1$, such that $du,dv\notin E(G)$ (by using \ref{x}), and then from \ref{xto}, $\{u,v,a,d,v_3,v_4\}$ induces a $P_2+P_4$; so $[A_2,D_4]=\es$.
     \item\label{F3-A4D1}  From \ref{DZ}, \ref{xto}, \ref{adt}, \ref{b}:\ref{tz} and from \ref{F3-A2D4}, we see that $\{v_1, v_4\} \cup A_2 \cup D_4 \cup Z_5$ and $\{v_2, v_5\} \cup A_4 \cup Z_1 \cup T$ are stable sets.
 \end{enumerate}
  Now from \ref{F3-W1}, \ref{F3-W2} and \ref{F3-W12}, and since $\omega(G)\geq 5$, we conclude that $\chi(G[W_1 \cup W_2]) = \max\{\chi(G[W_1]),$ $ \chi(G[W_2])\} \leq    \max\{\omega(G) - 2,~ 2\}\leq \omega(G) - 2$. Also from \ref{F3-A4D1}, we see that $\chi(G-(W_1 \cup W_2))\leq 2$. Hence $\chi(G)= \omega(G)$, and we are done. This proves \cref{F3}.
\end{proof}

\begin{lemma}\label{F5}
    Let $G$ be a ($P_2+P_4$, $K_4 -e$)-free graph. If $G$ contains an $F_4$, then  either $\chi(G)\leq 6$ or $\chi(G)= \omega(G)$.
\end{lemma}
\begin{proof}
    Let $G$ be a ($P_2+P_4$, $K_4 -e$)-free graph.  Suppose that $G$ contains an $F_4$. We label such an $F_4$ as shown in Figure \ref{Gi}, and we let $C := \{v_1, v_2, v_3, v_4, v_5\}$. Then with respect to $C$, we define the sets $A$, $B$, $D$, $Z$ and $T$ as in  \cref{genprop-1}, and we use the properties in  \cref{genprop-1}. Clearly     $b_1 \in B_1$. From \cref{F1,F2,F3}, we may assume that $G$ is $(F_1, F_2, F_3)$-free, and hence it follows that $(B\sm B_1) \cup A_3 \cup A_5 = \emptyset$. By \ref{x}, $B_1\cup Z_4$ is a clique of size at most $\omega(G) - 2$.  Moreover, we claim the following.

    \begin{claim}\label{F5-D1D2}
    We may assume that $D_1\cup D_2$ is an empty set.
    \end{claim}
    \no{\it Proof of \cref{F5-D1D2}}.~Suppose, up to symmetry, that $D_1\neq \es$. Then we have $B_1 = \{b_1\}$, by \ref{b}:\ref{b2d}. Now  we define the sets
        $S_1 := \{ v_3, v_5,b_1\} \cup A_1 \cup D_5 $,
        $S_2 :=  A_2 \cup D_3 \cup Z_4$,
        $S_3 :=  \{v_2\} \cup A_4 \cup T$,
        $S_4 := \{v_1, v_4\} \cup D_1$,
        $S_5 :=  D_2$ and
        $S_6 := D_4 \cup Z_1 \cup Z_2 \cup Z_3 \cup Z_5$ so that $V(G)= S_1\cup S_2\cup \cdots \cup S_6$. Then from \ref{DZ}, \ref{xto}, \ref{adt} and \ref{b}:\ref{da}, and from above arguments, we conclude that $S_j$ is a stable set, for  each $j\in \{1,2,\ldots, 6\}$, and hence  $\chi(G) \leq 6$. So we may assume that $D_1=\es$. Likewise,  we may assume that $D_2=\es$. $\sq$

Next:

  \begin{claim}\label{F5-B1geq4}
    We may assume that $|B_1|\geq 4$.
    \end{claim}
\no{\it Proof of \cref{F5-B1geq4}}.~Suppose that $|B_1| \leq 3$. If $B_1\sm \{b_1\}\neq \es$, then we let $B_1'$ consist of one vertex of $B_1\sm \{b_1\}$, otherwise we let $B_1'=\es$, and if $B_1'\neq \es$ and $B_1\sm (B_1'\cup \{b_1\})\neq \es$, then we let $B_1''$ consist of one vertex of $B_1\sm (B_1'\cup \{b_1\})$, otherwise we let $B_1''=\es$. Then we define the sets
    $S_1 :=  \{v_1\} \cup A_4 \cup T$,
        $S_2 := \{v_4,b_1\} \cup A_1 \cup A_2 $,
        $S_3 :=  \{v_3, v_5\} \cup B_1'\cup D_3 $,
       $S_4 :=  B_1'' \cup D_5$,
        $S_4 :=  \{v_2\} \cup D_4 \cup Z_1  \cup Z_3$         and
        $S_6 :=  Z_2 \cup Z_4 \cup Z_5$ so that $V(G)= S_1\cup S_2\cup \cdots \cup S_6$, by \cref{F5-D1D2}. Then from \ref{DZ}, \ref{xto}, \ref{adt}, \ref{b}:\ref{da} and from above properties, we conclude that  $S_j$ is a stable set, for each $j\in \{1,2,\ldots, 6\}$, and hence  $\chi(G) \leq 6$, and we are done.
         So we may assume that $|B_1|\geq 4$. $\sq$

    \medskip
   By using \cref{F5-B1geq4}, we may assume that there are vertices, say $x,y$ and $w$, in $B_1\sm \{b_1\}$.
    Then for any  $a \in A_4$, since $\{a, x, y, v_1\}$ does not induce a $K_4-e$, we have either $a x \notin E(G)$ or $ay \notin E(G)$; so we let  $A_4':= \{a \in A_4\mid ax\notin E(G)\}$ and $A_4'':= \{a \in A_4\sm A_4'\mid ay \notin E(G)\}$ so that  $A_4 := A_4' \cup A_4''$.
    Now we define the sets
        $S_1 := \{v_3, v_5,b_1\} \cup A_2 \cup D_3$,
        $S_2 :=  \{w\} \cup A_1 \cup D_5$,
        $S_3 :=  \{x\} \cup A_4'\cup Z_1$,
        $S_4 :=  \{y\} \cup A_4''$,
        $S_5 :=  \{v_1\} \cup D_4 \cup Z_2 \cup Z_5$,
        $S_6 := \{v_2,v_4\}  \cup Z_3\cup T$ and
        $M := Z_4 \cup (B_1\setminus\{b_1,x,y,w\})$  so that $V(G)= S_1\cup S_2\cup \cdots \cup S_6\cup M$. Then from \ref{DZ}, \ref{xto}, \ref{adt}, \ref{b}:\ref{da}, \ref{b}:\ref{tz} and from above arguments, we conclude that $S_j$ is a stable set, for each $j\in \{1,2,\ldots, 6\}$, and hence  $\chi(G-M) \leq 6$. Also since $\{v_1,v_2\}$ is complete to $M$, and since $M$ is a clique (by \ref{x}), clearly $M$ is a clique of size at most $\omega(G)-6$ , and hence $\chi(G)\leq \chi(G[M])+\chi(G-M)\leq (\omega(G)-6)+6=\omega(G)$. This proves \cref{F5}.
    \end{proof}
\begin{lemma}\label{C5}
    Let $G$ be a ($P_2+P_4$, $K_4 -e$)-free graph. If $G$ contains a  $C_5$, then   either $\chi(G)\leq 6$ or $\chi(G)= \omega(G)$.
\end{lemma}
\begin{proof}
    Let $G$ be a ($P_2+P_4$, $K_4 -e$)-free graph.  Suppose that $G$ contains a  $C_5$. We may assume such a  $C_5$ with the vertex-set, say $C := \{v_1, v_2, v_3, v_4, v_5\}$ and with the edge-set $\{v_1 v_2, v_2 v_3, v_3 v_4, v_4 v_5, $ $v_5 v_1\}$. Then with respect to $C$, we define the sets $A$, $B$, $D$, $Z$ and $T$ as in  \cref{genprop-1}, and we use the properties in  \cref{genprop-1}. From \cref{F5}, we may assume that $G$ is $F_4$-free, and hence it follows that $B = \emptyset$. Moreover, for any $a\in A_{i}$ and $d\in D_{i+1}$, since $\{a,d\}\cup (C\sm \{v_{i+1}\})$ does not induce  an $F_4$, we see that $ad\notin E(G)$; so $[ A_{i}, D_{i+1}]=\es$, for each $i\in \langle 5 \rangle$. Then we define the sets
        $S_1 :=  \{v_2, v_4\} \cup A_1 \cup D_2 \cup Z_3$,
        $S_2 :=   A_2 \cup  D_3$,
        $S_3 :=  A_3 \cup D_4 \cup Z_1$,
        $S_4 :=  \{v_3, v_5\} \cup A_4 \cup D_5$,
        $S_5 :=  \{v_1\} \cup A_5 \cup T$ and
        $S_6 :=  D_1 \cup Z_2 \cup Z_4 \cup Z_5$ so that $V(G)= S_1\cup S_2\cup \cdots \cup S_6$. Then from \ref{DZ}, \ref{xto} and  \ref{adt}, and from above arguments, we conclude that  $S_j$ is a stable set, for each $j\in \{1,2,\ldots, 6\}$, and hence  $\chi(G) \leq 6$. This proves \cref{C5}.
        \end{proof}

\begin{lemma}\label{C7}
    Let $G$ be a ($P_2+P_4$, $K_4 -e$, $C_5$)-free graph. If $G$ contains a  $C_7$, then    $\chi(G)\leq 3$.
\end{lemma}
\begin{proof}
Let $G$ be a ($P_2+P_4$, $K_4 -e$, $C_5$)-free graph. Suppose that $G$ contains a $C_7$ with the vertex-set, say $C := \{v_1, v_2, v_3, v_4, v_5, v_6, v_7\}$ and with the edge-set \{$v_1 v_2$, $v_2 v_3$, $v_3 v_4$, $v_4 v_5$, $v_5 v_6$, $v_6 v_7$, $v_7 v_1$\}.  Then we define the following sets for $i \in \langle  7 \rangle$.
    \begin{align*}
    Q_i &:= \{v \in V (G) \setminus C \,|\, N(v) \cap C = \{v_{i-1}, v_{i+1}\}\},\\
        X_i &:= \{v \in V (G) \setminus C \,|\, N(v) \cap C = \{v_{i}, v_{i+1}, v_{i+3}, v_{i-2}\}\},\\
    Y_i &:= \{v \in V (G) \setminus C \,|\, N(v) \cap C =    \{v_i, v_{i+1}, v_{i-2}\}\}, \\
   L_i &:= \{v \in V (G) \setminus C \,|\, N(v) \cap C =  \{v_i, v_{i+1}, v_{i+3}\}\}, \mbox{and }\\
     M &:= \{v \in V(G)\setminus C \,|\, N(v) \cap C = \emptyset\}.
    \end{align*}
Further we let $Q := \cup_{i=1}^7 Q_i$,  $X:= \cup_{i=1}^7 X_i$, $Y:= \cup_{i=1}^7 Y_i$ and $L:=\cup_{i=1}^7 L_i$. Then we claim the following.
\begin{claim} \label{C7-V}
    $V(G) = C \cup Q \cup X \cup Y \cup L  \cup M$.
\end{claim}
{\it Proof of \ref{C7-V}}.~It is enough to show that any $v\in V(G)\sm (C\cup M)$ is in $Q \cup X \cup Y \cup L$. Let $u\in  V(G)\sm (C\cup M)$.  Then $N(u) \cap C \neq \emptyset$, and say $v_i\in N(u)$ for some $i\in \langle  7 \rangle$. Then since $\{u,v_i, v_{i+2},v_{i+3}, v_{i-3},v_{i-2}\}$ does not induce a $P_2+P_4$, we have $u$ has a neighbor in $\{v_{i+2},v_{i+3},v_{i-3},v_{i-2}\}$.  Now if $v_{i+1}\in N(u)$, then since $\{v_{i-1}, v_i,v_{i+1},u\}$  or $\{v_i,v_{i+1},v_{i+2},u\}$  does not induce a $K_4-e$, we have  $v_{i-1},v_{i+2}\notin N(u)$, and then since   $\{v_{i-2}, v_{i-3}, v_{i+3}, u\}$ does not induce a $K_4-e$  or since one of $\{v_{i+1},v_{i+2},v_{i+3}, v_{i-3}, u\}$ or  $\{v_{i},v_{i-1},v_{i-2}, v_{i-3}, u\}$ does not induce a $C_5$, we may assume that $v_{i-3}\notin E(G)$, and   hence $u\in  X \cup Y \cup L$, and we are done. So  for any $i\in \langle  7 \rangle$, we may assume that if $v_i\in N(u)$, then $v_{i+1},v_{i-1}\notin N(u)$, and then since one of $\{u,v_i,v_{i-1},v_{i-2},v_{i-3}\}$, $\{u,v_i,v_{i+1},v_{i+2},v_{i+3}\}$ or  $\{u,v_{i+2},v_{i+3},v_{i-3},v_{i-2}\}$ does not induce  a $C_5$ and from the earlier argument, we see that $u\in Q$. This proves \cref{C7-V}. $\sq$

\begin{claim}\label{c7props}
For each $i \in \langle  7 \rangle$,  we have: $Q_i$, $X_i$, $Y_i$, $L_i$ and $M$ are stable sets, $[Q_i, Q_{i+3} \cup Q_{i-3}]$ is an empty set, and $[M, Q \cup Y \cup L]$ is an empty set.\end{claim}
{\it Proof of \ref{c7props}}.~Suppose there are adjacent vertices, say $u$ and $v$, in one of the mentioned sets.  If $u,v \in Q_i\cup X_i\cup Y_i\cup L_i$, then one of $\{u,v,v_{i-1},v_{i+1}\}$, $\{u,v,v_{i},v_{i+3}\}$ or $\{u,v,v_{i},v_{i-2}\}$  induces a  $K_4-e$ or $\{u,v,v_{i+3}, v_{i-3}, v_{i-2}\}$ induces a $C_5$, and  if $u,v \in M$, then $\{u,v, v_1, v_{2},v_{3}, v_{4}\}$ induces a $P_2+P_4$.
 If $u\in Q_i$ and $v \in Q_{i+3}\cup Q_{i-3}$, then one of $\{u,v_{i-1},v_{i-2}, v_{i-3}, v\}$ or $\{u,v_{i+1},v_{i+2}, v_{i+3}, v\}$ induces a $C_5$.
     If   $u\in M$ and $v \in Q_i$ (or $v\in Y_i$ or $v\in L_i$)  for some $i$, then one of $\{u, v, v_i, v_{i-1}, v_{i+2}, v_{i+3}\}$ or  $\{v_{i-1},v_{i-2}, u, v, v_{i+3}, v_{i+2}\}$  induces a $P_2+P_4$. So \cref{c7props} holds. $\sq$

     \begin{claim}\label{xyz}
      For each $i\in \langle  7 \rangle$,  we have the following: (a) If $X_i\neq \es$, then $X\sm X_i=\es$. (b) If $Y_i\neq \es$, then $Y\sm Y_i=\es$. Likewise, if $L_i\neq \es$, then $L\sm L_i=\es$.
\end{claim}
 {\it Proof of \ref{xyz}}.~We will prove for $i=1$.\\
 $(a)$:~Let $x\in X_1$.
  If there is a  vertex, say $x'\in X\sm X_1$, then we may assume, up to symmetry, that $x'\in X_2\cup X_3\cup X_4$, and then    $\{v_1,v_2,x,x'\}$, induces a $K_4-e$   or one of $\{x,v_2,x',v_5,v_6\}$ or $\{x,v_2,v_3,x',v_6\}$  induces a $C_5$; so  $X\sm X_1=\es$.   This proves $(a)$.

  \smallskip
 \no{$(b)$}:~Let $y\in Y_1$. Suppose there is a  vertex, say $y'\in Y\sm Y_1$. Now if $yy'\in E(G)$, then one of $\{v_1,v_2,y,y'\}$, $\{v_5,v_6,y,y'\}$ or $\{v_6,v_7,y,y'\}$ induces a $K_4-e$, and if $yy'\notin E(G)$, then one of $\{v_4,v_5,y,v_2,y',v_7\}$, $\{v_4,y',v_2,y,v_6,v_7\}$, $\{y,v_2,v_7,y',v_4,v_5\}$  or  $\{v_3,v_4,y,v_6,v_7,y'\}$ induces a $P_2+P_4$ or one of $\{v_2,y',v_5,$ $ v_6,y\}$ or $\{v_2,v_3,y',v_6,y\}$ induces a $C_5$. This shows that $Y\sm Y_1=\es$. Likewise, if $L_1\neq \es$, then $L\sm L_1=\es$. This proves $(b)$. $\sq$

\smallskip
Next:

    \begin{claim}\label{c7-xemp}
    We may assume that $X$ is an empty set.
    \end{claim}
  {\it Proof of \ref{c7-xemp}}.~Suppose that $X \neq \emptyset$. Then there is an index $i \in \langle  7 \rangle$  such that $X_i \neq \emptyset$, say $i=1$, and let $x\in X_1$. So from \cref{xyz}:$(a)$, we have $X\sm X_1=\es$. Also  the following hold.
  \begin{enumerate}[label=($\roman*$)]\itemsep=0pt

\item If there is a  vertex, say $z\in L\setminus L_6$, then   one of $\{v_1,v_2,x,z\}$, $\{v_1,v_4,x,z\}$, $\{v_3,v_4,x,z\}$ or $\{v_4,v_5,x,z\}$ induces a $K_4-e$  or one of $\{x,v_2,z,v_5,v_6\}$, $\{x,v_6,z,v_3,v_2\}$, $\{x,v_4,z,v_7,v_1\}$, $\{x,v_4, v_5, z,$ $v_1\}$ or
    $\{x,v_1,z,v_3,v_4\}$ induces a $C_5$; so  $L\sm L_6=\es$. Likewise, $Y\sm Y_3=\es$.

   \item   If there is a  vertex, say $q\in Q_1 \cup Q_2$, then $\{v_1, v_2, q, x\}$ induces $K_4 - e$ or one of $\{q,v_2, x, v_6, v_7\}$ or $\{q,v_1, x, v_4, v_3\}$ induces a $C_5$; so $Q_1 \cup Q_2=\es$.
   \item If there are adjacent vertices, say $u, v\in Q_5\cup X_1\cup Y_3\cup L_6$, then using \cref{c7props}, we see that one of  $\{u,v, v_4, v_6\}$,
   $\{u,v,v_3,v_4\}$ or $\{u,v,v_6,v_7\}$ induces a $K_4 - e$ or $\{u,v,v_4,v_5,v_6\}$ induces a $C_5$; so $Q_5\cup X_1\cup Y_3\cup L_6$ is a stable set.

 \end{enumerate}
From above arguments, we conclude that $V(G) = C \cup Q_3 \cup Q_4 \cup Q_5 \cup Q_6 \cup Q_7\cup X_1 \cup Y_3 \cup L_6 \cup M$.
      Now we define the sets
        $S_1 :=  \{v_5\} \cup Q_5 \cup X_1 \cup Y_3 \cup L_6 $,
        $S_2 :=   \{v_2, v_4, v_7\} \cup Q_4 \cup Q_7 \cup M$  and
        $S_3 :=  \{v_1, v_3, v_6\} \cup Q_3 \cup Q_6$ so that $V(G)= S_1\cup S_2 \cup S_3$.
        Then from \cref{c7props} and  $(iii)$, $S_j$ is a stable set for each $j\in \{1,2,3\}$, and so $\chi(G) \leq 3$. Hence we may assume that $X = \emptyset$. $\sq$

\smallskip
Next:
          \begin{claim}\label{c7-yemp}
    We may assume that $Y$ is an empty set. Likewise,   we may assume that $L$ is an empty set.
    \end{claim}
  {\it Proof of \ref{c7-yemp}}.~Suppose that $Y \neq \emptyset$. Then there is an index $i \in \langle  7 \rangle$  such that $Y_i \neq \emptyset$, say $i=1$, and let $y\in Y_1$. So from \cref{xyz}:$(b)$, we have $Y\sm Y_1=\es$. Also  the following hold.
  \begin{enumerate}[label=($\roman*$)]\itemsep=0pt
      \item\label{c7y-d14}  If there is a  vertex, say $q\in Q_1$, then $\{v_1, v_2, q, y\}$ induces a $K_4 - e$ or  $\{q,v_2, y, v_6, v_7\}$ induces a $C_5$; so $Q_1=\es$. Likewise, if $L_3\neq \es$, then $Q_4=\es$.

     \item\label{c7y-z126} If there is a vertex, say $z\in L_1 \cup L_2 \cup L_6$, then   $\{y, z, v_1, v_2\}$  induces a $K_4 - e$ or one of $\{y, z, v_4,v_5, v_6\}$ or $\{y, v_2, z, v_5, v_6\}$ induces a $C_5$  or $\{ v_4, v_5,y, v_1, v_7, z\}$ induces a $P_2 + P_4$; so $L_1 \cup L_2 \cup L_6=\es$.
     \item\label{c7y-yd} If there are adjacent vertices, say $u \in Y_1$ and $v\in Q_2 \cup Q_3 \cup Q_5 \cup Q_7$, then  $\{v_1, v_2, u,v\}$ induces a $K_4 - e$ or $\{u, v_2, v_3, v_4, v\}$ induces a $C_5$; so $[Y_1, Q_2 \cup Q_3 \cup Q_5 \cup Q_7] = \emptyset$.

  \item\label{c7y-zd} For any $i\in \langle  7 \rangle$  and for any $z \in L_{i+1}$ and $q\in Q_i\cup Q_{i+3}$, since $\{v_{i+1},v_{i+2},z,q\}$ does not induce a $K_4 - e$, we have $[L_{i+1}, Q_i\cup Q_{i+3}]=\es$. Also, for any $p\in L_7$ and $y'\in Y_1$, since $\{v_1,v_2,y',p\}$ does not induce a $K_4-e$, we have $[L_7,Y_1]=\es$, and hence for any $p'\in L_7$ and
      $q'\in Q_5$, since $\{v_1, y, v_6, q', p'\}$ does not induce a  $C_5$ (by using \ref{c7y-yd}), we see that $[Q_5,L_7]=\es$.
     \end{enumerate}
     Now if $L_3=\es$, then we define the sets  $S_1 := \{v_5\} \cup Q_2 \cup Q_5 \cup Y_1 \cup L_7$,
        $S_2 :=   \{v_1, v_3, v_6\} \cup Q_3 \cup Q_6 \cup L_4$  and
        $S_3 :=  \{v_2, v_4, v_7\} \cup Q_4 \cup Q_7 \cup L_5\cup M$, and if  $L_3\neq\es$ (so $L\sm L_3=\es$ (by \cref{xyz}), and $Q_1\cup Q_4=\es$ (by \ref{c7y-d14})), then we define the sets    $S_1 := \{v_3, v_5, v_7\} \cup Q_3 \cup Q_7 \cup Y_1$,
        $S_2 :=   \{v_2\} \cup Q_2 \cup Q_5 \cup L_3$, and
        $S_3 :=  \{v_1, v_4, v_6\} \cup Q_6\cup M$ so that  $V(G)= S_1\cup S_2 \cup S_3$.
        Then from \cref{c7props},  \ref{c7y-yd} and \ref{c7y-zd}, we conclude that in both cases, $S_j$ is a stable set for each $j\in \{1,2,3\}$, and so $\chi(G) \leq 3$. Hence we may assume that $Y = \emptyset$. Likewise,   we may assume that $L=\es$. $\sq$

        \medskip
         From \cref{c7-xemp} and \cref{c7-yemp}, we see that $V(G) = C \cup Q\cup M$. If $Q_i = \emptyset$, for some $i\in \langle  7 \rangle$, say $i=1$, then we define the sets  $S_1 := \{v_1, v_3, v_6\} \cup Q_3 \cup Q_6$,
    $S_2 :=   \{v_2, v_5\} \cup Q_2 \cup Q_5$  and
    $S_3 :=  \{v_4, v_7\} \cup Q_4 \cup Q_7\cup M$, and then from \ref{c7props},  $S_j$ is a stable set for  each $j\in \{1,2,3\}$ and so $\chi(G)\leq 3$. So we may assume that
      $Q_i \neq \emptyset$, for each $i \in \langle  7 \rangle$. Then  for any $p\in Q_i$, $q\in Q_{i+1}$ and $r\in Q_{i+2}$, since  $\{v_{i-2},v_{i-3}, v_{i+1},v_{i}, p,q,r\}$ does not induce  a $P_2 + P_4$, we have $pq, qr\in E(G)$, and then since $\{v_{i+1},p,q,r\}$ does not induce a $K_4-e$, we have $pr\notin E(G)$; so $[Q_i, Q_{i+2}]=\es$ for each $i\in \langle  7 \rangle$.  Now we define the sets
    $S_1 := \{v_1, v_3, v_5\} \cup Q_1 \cup Q_3 \cup Q_5$,
    $S_2 :=   \{v_2, v_4, v_6\} \cup Q_2 \cup Q_4 \cup Q_6$, and
    $S_3 :=  \{v_7\} \cup Q_7\cup M$.
    Then from \ref{c7props} and from above arguments, we see that $S_j$ is a stable set  for each $j\in \{1,2,3\}$, and so $\chi(G) \leq 3$. This proves \cref{C7}.
     \end{proof}

      \medskip
      \no{\bf Proof of \cref{main-c5c7}}.
 Let $G$ be a ($P_2+P_4$, $K_4 -e$)-free graph which is not a perfect graph. Since an odd hole $C_{2t+1}$, where $t\geq 4$, contains a $P_2+P_4$, and since an odd antihole $\overline{C_{2p+1}}$, where $p\geq 3$, contains a $K_4-e$, by the `Strong perfect graph theorem' \cite{spgt}, we may assume that $G$ contains a $C_5$ or a $C_7$. Now if $G$ contains a $C_5$, then the proof follows from \cref{C5}. So we may assume that $G$ is $C_5$-free and contains a $C_7$. Then the proof follows from \cref{C7}. This proves \cref{main-c5c7}. \hfill{$\Box$}

\section{Proof of \cref{mainthm-2}}\label{sec:w4}
In this section, we present a proof of \cref{mainthm-2}. We split the proof into two parts based on whether our graph contains an $F$ or not (see \cref{fig} for the graph $F$), and they are given in \Cref{sec:F2,sec:F2-free} respectively. First we prove some useful  lemmas which we use often in the subsequent sections.

 \begin{figure}[h]
\centering
 \includegraphics[width=15cm]{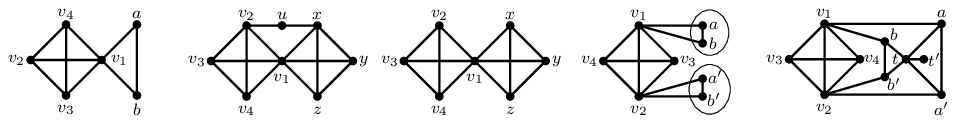}
\caption{Graphs $F$, $H_1'$, $H_1$, $H_2$ and $H_3$ (left to right).   In $H_2$, the edges between $\{a,b\}$ and $\{a',b'\}$ are arbitrary. }\label{fig}
\end{figure}

 Let $G$ be a bad connected ($P_2+P_4$, $K_4-e$)-free graph with $\omega(G) = 4$. Suppose that $G$ contains a $K_4$   with vertex-set, say $C := \{v_1, v_2, v_3, v_4\}$. Since $G$ is $(K_4 - e, K_5)$-free, every vertex in $ V(G)\setminus C$ has at most one neighbor in $C$. For each $i \in \{1,2,3,4\}$, we let $A_i := \{ u \in V(G)\setminus C \mid N(u) \cap C = \{v_i\} \}$, $A := \cup_{i=1}^4 A_i$, and let $T := \{u \in V(G)\setminus C\mid N(u) \cap C = \emptyset\}$; so $V(G)=C\cup A\cup T$. Since $G$ is bad, we have $deg(u)\geq 4$ for each $u\in V(G)$, and hence we may assume that each $A_i$ is nonempty (otherwise $deg(v_i) = 3$ and hence $G$ is a  good  graph), and we let  $a_i^*\in A_i$, $i\in  \{1,2,3,4\}$. Moreover, the sets $A$ and $T$ satisfy the following properties given in \Cref{lem:A,lemmat} below.

\begin{lemma}\label{lem:A}
For  $i \in \{1,2,3,4\}$, the following hold:
\begin{enumerate}[label=  ($\mathbb{M}\arabic*$), leftmargin=1.05cm] \itemsep=0pt
\item \label{lemAi}
Each $A_i$ is a disjoint union of cliques of size at most $3$.
\item \label{F1-obs}
Suppose that $A_i$ contains a clique of size at least 2, say $K$. Then any vertex in $V(G)\setminus (A_i\cup C)$ has  at most one neighbor in $K$.  
\item \label{lemAij}
For $j \in \{1,2,3,4\}$ and $j\neq i$, $G[A_i \cup A_j\cup T]$ is $P_4$-free and hence perfect.
\item\label{lem-cru} For $j \in \{1,2,3,4\}$ and $j\neq i$, suppose that there are vertices, say $a,b\in A_i$ and $a',b'\in A_j$ such that $aa', bb'\in E(G)$ and $ab,a'b',ab',ba'\notin E(G)$. If $t\in T$ has a neighbor in $\{a,b,a',b'\}$, then $t$ is complete to $\{a,b,a',b'\}$.
\end{enumerate}
\end{lemma}
 \begin{proof}\ref{lemAi}:~Since $G$ is ($K_4-e, K_5$)-free and since each $A_i$ is complete to $\{v_i\}$,   each $A_i$  induces a $(P_3, K_4)$-free graph, and so it is a disjoint union of cliques of size at most $3$, by \ref{p3-free}. $\sq$

 \smallskip
\noindent{\ref{F1-obs}}:~Otherwise, for any $p\in V(G)\setminus (A_i\cup C)$,   $\{p,v_i\}\cup K$  induces a $K_4-e$. $\sq$

\smallskip
\noindent{\ref{lemAij}}:~Since  $G[A_i \cup A_j \cup T \cup (C\setminus \{v_i,v_j\})]$ does not induce a $P_2+P_4$,  clearly $G[A_i \cup A_j\cup T]$ is $P_4$-free and hence perfect (by \ref{p4-free}). $\sq$

\smallskip
\noindent{\ref{lem-cru}}:~We prove this when $i=1$ and $j=2$. Suppose, up to symmetry, that $ta\in E(G)$. If $tb\notin E(G)$, then  $\{t, a, b,b',v_2,v_3\}$ or $\{v_3,v_4,b, b', t,a\}$ induces a $P_2+P_4$; so $tb\in E(G)$. But then since $\{v_3,v_4,a, t, b, a',b'\}$ does not induce a $P_2+P_4$, we conclude that $t$ is complete to $\{a,a',b,b'\}$. This proves \ref{lem-cru}.
\end{proof}

\begin{lemma}\label{lemmat}
	$G[T]$ is isomorphic to $K_2$ or $T$ is a  stable set.
\end{lemma}
\begin{proof}By \ref{p3-free}, it is enough to show that $G[T]$ is ($K_3, P_3,  K_1+K_2$)-free, and is shown below.
\begin{claim}\label{tk3}
		$G[T]$ is $K_3$-free.
	\end{claim}
	\noindent{\it Proof of \cref{tk3}}.~Suppose to the contrary that $G[T]$ contains a $K_3$, say with vertex-set $K:=  \{x, y, z\}$. Then we claim that each $a_i^*$ has a neighbor in $K$, $i\in  \{1,2,3,4\}$.   Suppose there is an index $i\in  \{1,2,3,4\}$, say $i=1$, such that $a_1^*$ is anticomplete to $K$. Then
since $\{x,y,z,a_1^*, a_2^*, v_1,v_2,v_3\}$ does not induce a $P_2+P_4$, $a_2^*$ is complete to $K$, by \cref{obs1}(i). Likewise, $a_3^*$ is also complete to $K$. Thus $\{a_2^*, a_3^*\}$ is complete to $K$ which is a contradiction to  \cref{obs1}(ii); so each $a_i^*$ has a neighbor in $K$.
Then  again by using \cref{obs1}, we may assume that there is an index $j\in  \{1,2,3,4\}$ such that $|N(a_j^*)\cap K| = 1 = |N(a_{j+1}^*)\cap K|$, say $j=1$. Now if $N(a_1^*)\cap K =  N(a_{2}^*)\cap K=\{x\}$ (say), then  $\{y,z, a_1^*, a_2^*, v_1,v_2,v_3\}$ induces a $P_2+P_4$, and if $N(a_1^*)\cap K = \{x\}$ and $N(a_{2}^*)\cap K=\{y\}$ (say), then  $\{v_3,v_4,a_1^*,x, y, a_2^* \}$ or $\{v_3,v_4, a_1^*, a_2^*, y, z\}$ induces a $P_2+P_4$ which are contradictions. So \cref{tk3} holds. $\sq$

	\begin{claim}\label{tp3}
		$G[T]$ is $P_3$-free.
	\end{claim}
	\noindent{\it Proof of \cref{tp3}}.~Suppose to the contrary that $G[T]$ contains a $P_3$, say with vertex-set $\{x,y,z\}$ and edge-set $\{xz,zy\}$. Then since $x$ and $y$ are not comparable vertices in $G$,  there are vertices, say $x'$ and $y'$, such that $xx', yy'\in E(G)$ and $xy', yx'\notin E(G)$. Now if $x'\in T$, then from \cref{tk3}, we have $zx'\notin E(G)$  and then $\{v_1, v_2, x', x, z, y\}$ induces a $P_2 + P_4$; so $x'\notin T$. Likewise, $y'\notin T$. Thus $\{x', y'\} \subset A$, and let $x'\in A_1$. If $y'\in A_2$,  then $\{v_{3}, v_{4}, x, x', y', y\}$  or $\{x,x', y, y', v_2, v_{3}\}$ induces a $P_2+P_4$; so $y'\notin A_2$. Likewise, $y'\notin A_3\cup A_4$. Hence, $y'\in A_1$. Then since   $\{v_2, v_3, x,x',y',y\}$, $\{v_2, v_3, x',x,z,y\}$ and $\{v_2, v_3, y',y,z,x\}$ do not induce   $P_2 + P_4$'s, we have $x'y'\notin E(G)$ and $x'z,y'z\in E(G)$. Then since $\{y,y', a_2^*, a_3^*, v_2,v_3,v_4\}$ does not induce a $P_2 + P_4$, we may assume that $a_2^*$ has a neighbor in $\{y,y'\}$.
 Then since $\{x,x', y,y', a_2^*, v_2,v_3\}$ does not induce a $P_2 + P_4$, we see that  $a_2^*$ has a neighbor in $\{x,x'\}$, and then since
 $\{x,x', z, y,y', a_2^*\}$ does not induce a $K_4-e$, we have $a_2^*z\notin E(G)$. But then by using \cref{obs1}:(i), we conclude that $\{v_3,v_4, x,x',   y,y', a_2^*\}$ induces a $P_2 + P_4$ which is a contradiction.   So \cref{tp3} holds. $\sq$

\begin{claim}\label{t2k2}
		$G[T]$ is $2K_2$-free.
	\end{claim}
	\noindent{\it Proof of \cref{t2k2}}.~Suppose to the contrary that $G[T]$ contains a $2K_2$, say with vertex-set $\{x,x', y,y'\}$  where $xx', yy'\in E(G)$. We first observe that  if there exists an $i \in \{1,2,3,4\}$ such that $N(a_{i}^*) \cap \{x,x',y,y'\} \neq \emptyset$, then $a_i^*$ is complete to $\{x,x',y,y'\}$. (Otherwise, say $i=1$ and then $\{x,x', y,y',a_1^*,v_1,v_2\}$ or $\{v_3,v_4,a_1^*,x,x',$ $ y,y'\}$  induces a $P_2 + P_4$). Now since $\{x,x', a_1^*, a_2^*, v_1,v_2,v_3\}$ does not induce a $P_2 + P_4$, we may assume that $a_1^*$ has a neighbor in $\{x,x'\}$, and hence $a_1^*$ is complete to $\{x,x',y,y'\}$. Likewise,  since $\{x,x', a_2^*, a_3^*, v_2,v_3,v_4\}$ does not induce a $P_2 + P_4$, we may assume that $a_2^*$ has a neighbor in $\{x,x'\}$, and hence $a_2^*$ is complete to $\{x,x',y,y'\}$. But then $\{a_1^*,x,x',a_2^*\}$ or $\{a_1^*,x,y,a_2^*\}$ induces a $K_4-e$ which is a contradiction. So  \ref{t2k2} holds. $\sq$
	
	\begin{claim}\label{tk2uk1}
		$G[T]$ is ($K_1 + K_2$)-free.
	\end{claim}
	\noindent{\it Proof of \cref{tk2uk1}}.~Suppose to the contrary that $G[T]$ contains a $K_1+K_2$, say with vertex-set $\{t,x,x'\}$  where $xx'\in E(G)$. Then the following hold:
   \begin{enumerate}[label=($\roman*$),leftmargin=0.9cm]\itemsep=0pt
   \item\label{tk2uk1-1} If $a\in N(t)\cap A_i$, then $a$ is complete to $\{x,x'\}$; otherwise, say $i=1$ and then $\{x,x',t,a,v_1,v_2\}$ or $\{v_3,v_4,t,a,x,x'\}$  induces a $P_2 + P_4$.
   \item\label{tk2uk1-2} Since $G$ is connected and from \ref{t2k2},  $t$ has a neighbor in $A_i$ for some $i$, say $i=1$, and let $a_1\in N(t)\cap A_1$. Then   $a_1$ is complete to $\{t,x,x'\}$, by \ref{tk2uk1-1}.

\item\label{tk2uk1-3} for each $j\neq 1$, we have $a_j^*t\notin E(G)$; otherwise,   $\{a_1,a_j^*,t,x,x'\}$ induces a $K_4-e$, by \ref{tk2uk1-1} and \ref{tk2uk1-2}.

\item\label{tk2uk1-4} Since $\{x,x',a_2^*,a_3^*,v_2,v_3,v_4\}$ does not induce a  $P_2 + P_4$, we may assume that $a_2^*$ has a neighbor in $\{x,x'\}$, and then by \ref{tk2uk1-3}, since  $\{v_3,v_4,t,a_1,a_2^*,x,x'\}$ does not induce a  $P_2 + P_4$ or $K_4-e$, we see that $a_2^*$ is complete to $\{x,x'\}$.  Also since $\{x,x',a_3^*,a_4^*,v_3,v_4,v_1\}$ does not induce a  $P_2 + P_4$, we may assume that $a_3^*$ has a neighbor in $\{x,x'\}$, and then as in the previous argument, we see that $a_3^*$ is complete to $\{x,x'\}$.
       \end{enumerate}
Thus from \ref{tk2uk1-2} and \ref{tk2uk1-4}, we conclude that $\{a_1,a_2^*,a_3^*\}$ is complete to $\{x,x'\}$, and so $\{a_1,a_2^*,a_3^*,x,x'\}$ induces a $K_5$ or $K_4-e$ which is a contradiction. So \cref{tk2uk1} holds. $\sq$

  Now the lemma follows from \cref{tk3}, \cref{tp3} and from \cref{tk2uk1}.
\end{proof}

{\it From now on, whenever our bad graph $G$ contains  a $K_4$   say with vertex-set $\{v_1, v_2, v_3, v_4\}$, then we define the sets $C, A$ and $T$ as in the beginning of this section and we use   \Cref{lem:A,lemmat}.}

\medskip
Next, we prove that if our  graph $G$ contains an $H_1$, then $\chi(G)=4$. First we prove the following.

\begin{lemma}\label{lemF11}
 If $G$ contains an  $H_1'$,	then $\chi(G)= 4$.
\end{lemma}
\begin{proof}
Suppose that $G$ contains an $H_1'$. We  label such an $H_1'$ as shown in Figure~\ref{fig}.  Clearly  $\{x,y,z\} \subset A_1$ and $u\in A_2$. Then we claim the following.
 \begin{claim}\label{a1}
$A_1\setminus \{x,y\}$ is a stable set.
\end{claim}
{\it Proof of \cref{a1}}.~Suppose to the contrary that there are adjacent vertices, say $a$ and $b$, in $A_1\setminus \{x,y\}$. Then from \ref{lemAi}, we have $z\notin \{a,b\}$ and that  $\{a,b\}$ is anticomplete to $\{x,y,z\}$. Now since $\{a,b,x,u,v_2,v_3\}$ does not induce a $P_2 + P_4$, we may assume that $au\in E(G)$, and then  $\{v_3,v_4,a,u,x,y\}$ induces a $P_2 + P_4$ which is a contradiction.
 So \cref{a1} holds.  $\sq$

 \begin{claim}\label{xA3}
 $A_3$ is anticomplete to  $x$.
\end{claim}
 {\it Proof of \cref{xA3}}.~For any $a\in A_3$,  if $ax\in E(G)$, then by \ref{F1-obs}, $ay,az\notin E(G)$, and then $\{y,z,u,v_2,v_3,v_4,a\}$ induces a $P_2 + P_4$; $\{x\}$ is anticomplete to $A_3$. This proves \cref{xA3}.  $\sq$

\begin{claim}\label{Ajind}
 For $j\neq 1$, $A_j$ is a stable set.
\end{claim}
{\it Proof of \cref{Ajind}}.~First if $N(u)\cap A_2\neq \es$, say $w\in N(u)\cap A_2$, then by using \ref{F1-obs}, we see that $xw\notin E(G)$ and that $w$ has at most one neighbor in $\{y,z\}$, but then $\{v_3,v_4,w,u,x,y,z\}$ induces a $P_2 + P_4$; so we may assume that $N(u)\cap A_2= \es$. Now suppose to the contrary that there are adjacent vertices, say $a$ and $b$, in $A_2$.  Then $ua, ub\notin E(G)$, and then  since $\{a,b,u,x,v_1,v_4\}$ does not induce a $P_2 + P_4$, from \ref{F1-obs},    $x$ has at most one neighbor in $\{a,b\}$.
But then $\{v_3,v_4,u,x,a,b\}$ induces a $P_2 + P_4$ which is a contradiction.  So $A_2$ is a stable set.

 Next suppose to the contrary that there are adjacent vertices, say $a'$ and $b'$, in $A_3$. Recall that  $a'x,b'x\notin E(G)$, by \cref{xA3}. Then since $\{a',b',u,x,v_1,v_4\}$ does not induce a $P_2 + P_4$, by using \ref{F1-obs}, we may assume that $ua'\in E(G)$ and $ub'\notin E(G)$. Again since $\{y,z,u,v_2,v_3, b'\}$
does not induce a $P_2 + P_4$, we may assume that $b'y\in E(G)$ and $b'z\notin E(G)$, by \ref{F1-obs}. But then by \ref{F1-obs}, we have $a'y\notin E(G)$ and then $\{v_2, v_4, a',b',y,x\}$ induces a $P_2 + P_4$ which is a contradiction. So  $A_3$ is a stable set. Likewise, $A_4$ is a stable set. This proves  \cref{Ajind}. $\sq$

 \begin{claim}\label{F1TA2}
   $A_2\cup\{y\}\cup  T$ is a stable set.
 \end{claim}
{\it Proof of \cref{F1TA2}}.~
Let $t\in T$ and $a\in A_2$ be arbitrary. Now, if $ay\in E(G)$, then $a\neq u$ and then by \ref{F1-obs} and \cref{Ajind}, we see that $\{v_3,v_4,a,y,x,u\}$ induces a $P_2 + P_4$; so  $A_2$ is anticomplete to $\{y\}$. Likewise,   $A_2$ is anticomplete to $\{z\}$.  Next, if $ty\in E(G)$, then $\{v_3,v_4,t,y,x,u\}$ or $\{v_3,v_4,t,u,x,z\}$ induces a $P_2 + P_4$, by \ref{F1-obs};  so $T$ is anticomplete to $\{y\}$. Likewise, $T$ is anticomplete to $\{z\}$. Also if $at\in E(G)$, then   $\{y,z,t,a,v_2,v_3\}$ induces a $P_2 + P_4$; so $T$ is anticomplete to $A_2$.  Thus the sets $\{y\}$, $T$ and $A_2$ are  anticomplete to each other. Next if there are adjacent vertices, say $t$ and $t'$, in $T$, then by  the first assertion,  $\{t,t',y,v_1,v_2,u\}$ induces a $P_2 + P_4$; so $T$ is a stable set. Now the claim follows from  \cref{Ajind}.  $\sq$

 \smallskip

Now we define the  sets $S_1:= \{v_1\} \cup A_4$, $S_2:= \{v_2\} \cup (A_1\setminus \{x,y\})$, $S_3:= \{v_3, y\} \cup A_2 \cup T$ and $S_4:= \{v_4, x\} \cup A_3$ so that $V(G) = S_1 \cup S_2 \cup S_3 \cup S_4$. Then by  above claims, we conclude that $S_{\ell}$ is a stable set for each $\ell\in \{1,2,3,4\}$ and hence $\chi(G) = 4$.
\end{proof}

\begin{lemma}
  \label{ai3}
	If $G$ contains an  $H_1$,	then $\chi(G)=4$.
\end{lemma}
\begin{proof}
Suppose that $G$ contains an $H_1$. We label such an $H_1$ as shown in Figure~\ref{fig}. Clearly  $\{x,y,z\} \subset A_1$.  By \cref{lemF11}, we may assume that $G$ is $H_1'$-free. Moreover, we claim the following.

 \begin{claim}\label{F1'-xyz-A234}
 $A_2 \cup A_3 \cup A_4$ is anticomplete to $\{x,y,z\}$.
 \end{claim}
 {\it Proof of \cref{F1'-xyz-A234}}.~Suppose to the contrary and  up to symmetry  that there is a vertex, say $a\in A_2$, such that $ax\in E(G)$. Then by \ref{F1-obs}, we have $ay,az\notin E(G)$ and then   $\{v_1,v_2,v_3,v_4,x,y,z,a\}$ induces an $H_1'$ which is a contradiction. So \cref{F1'-xyz-A234} holds. $\sq$

\begin{claim} \label{aiaj}
	For $i, j\in \{1,2,3,4\}$ and $i\neq j$, $A_i$ is anticomplete to $A_j$.
\end{claim}
{\it Proof of \cref{aiaj}}.~From \cref{F1'-xyz-A234} and by using symmetry, it is enough to show that $A_2$ is anticomplete to $(A_1\setminus \{x,y,z\})\cup A_3$. Now if there are adjacent vertices, say $a\in A_2$ and $b\in (A_1\setminus \{x,y,z\})\cup A_3$, then from \ref{lemAi} and from \cref{F1'-xyz-A234}, we see that $\{x,y,b,a,v_2,v_3\}$ or  $\{x,y,a,b,v_3,v_4\}$ induces a $P_2 + P_4$ which is a contradiction; so \cref{aiaj} holds. $\sq$

\begin{claim} \label{tai}
$T$ is anticomplete to $ A_2 \cup A_3 \cup A_4$, and $T$ is a stable set.
\end{claim}
{\it Proof of \cref{tai}}.~Suppose to the contrary and up to symmetry  that there are vertices, say $a\in A_2$ and $t\in T$, such that $at\in E(G)$.   Then from \ref{F1-obs} and \cref{F1'-xyz-A234}, we conclude that  $\{x,y,z,t,a, v_2, v_3\}$ induces a $P_2 + P_4$ which is a contradiction; so $T$ is anticomplete to $ A_2 \cup A_3 \cup A_4$. Next if $uv$ is an edge in $G[T]$, then by the first assertion,
$\{u,v,a_2^*,a_3^*,v_2,v_3,v_4\}$ induces a $P_2 + P_4$; so $T$ is a stable set. This proves \cref{tai}. $\sq$

\smallskip

 Now we give a $4$-coloring of $G$ using colors from the set $\{1,2, 3, 4\}$  as  follows:
\begin{itemize}\itemsep=0pt
\item We color $v_1, v_2, v_3$ and  $v_4$ with colors $1, 2, 3$ and $4$ respectively.
 \item By using \ref{lemAi} and \cref{aiaj}, we color the vertices in $A_1$ using the colors from $\{2, 3, 4\}$,
we color the vertices in $A_2$ using the colors from $\{1, 3, 4\}$,
  we color  the vertices in $A_3$ using the colors from  $\{1, 2, 4\}$, and we color  the vertices in $A_4$ using the colors from $\{1, 2, 3\}$.
    \item By using \ref{tai}, we color  the vertices in $T$ with color $1$.
\end{itemize}
Hence $\chi(G) =4$. This proves \cref{ai3}.
\end{proof}

\subsection{$G$ contains an $F$}\label{sec:F2}
In this section, we show that if $G$  contains an $F$, then  $\chi(G) = 4$. First we prove the following intermediate lemma.

\begin{lemma}\label{ai22}
	If $G$ contains an  $H_2$,	then $\chi(G) = 4$.
\end{lemma}
\begin{proof}
Suppose $G$ contains an $H_2$. We label such an $H_2$ as shown in Figure~\ref{fig}.  Clearly  $a,b\in A_1$ and $a', b' \in A_2$. We may assume that $G$ is $H_1$-free, by \cref{ai3}.  Further we claim the following.
\begin{claim}\label{F21-Aij}
For $i\neq j$, $G[A_i \cup A_j]$ is ($P_4, K_3$)-free and hence $\chi(G[A_i \cup A_j\cup (C\setminus \{v_i,v_j\})])= 2$.
\end{claim}
{\it Proof of \cref{F21-Aij}}.~By \ref{lemAij}, $G[A_i \cup A_j\cup T]$ is   perfect, for $i\neq j$. Moreover, since $G$ is $H_1$-free, we have $\omega(G[A_i])\leq 2$ for each $i$, and so by \ref{F1-obs}, $G[A_i \cup A_j]$ is $K_3$-free, for $i\neq j$.  Thus  for $i\neq j$,  $G[A_i \cup A_j]$ is a perfect $K_3$-free graph and so   $\chi(G[A_i \cup A_j])\leq 2$. Then since $(C\setminus \{v_i,v_j\})$ is anticomplete to $A_i \cup A_j$, we see that  $\chi(G[A_i \cup A_j\cup (C\setminus \{v_i,v_j\})])= 2$. This proves \cref{F21-Aij}. $\sq$

\begin{claim}\label{F21-K32}
There is no $K_3$ in $G[A_1 \cup A_2 \cup T]$ which has  two vertices from $A_1 \cup A_2$.
\end{claim}
{\it Proof of \cref{F21-K32}}.~Suppose to the contrary that there is a $K_3$ in $G[A_1 \cup A_2 \cup T]$ with vertices, say  $x, y$ and $z$, which has  two vertices from $A_1 \cup A_2$, say $x$ and $y$. Then by \cref{F21-Aij}, \ref{lemAi}, \ref{F1-obs} and by symmetry,  we have two cases: $(i)$~$x=a$, $y=a'$ and $z\in T$, and $(ii)$~$x\in A_1$, $y\in A_2\setminus \{a',b'\}$ and $z\in T$.

First suppose that $x=a$, $y=a'$ and $z\in T$. Then since $\{v_3, v_4, b, x, y, b'\}$ does not induce a $P_2 + P_4$, we have $bb' \in E(G)$, by \ref{F1-obs}. Again by \ref{F1-obs}, we have $zb, zb'\notin E(G)$. But then $\{v_3, v_4, z, x, b, b'\}$ induces a $P_2 + P_4$ which is a contradiction.
	
 Next suppose that $x\in A_1$, $y\in A_2\setminus \{a',b'\}$ and $z\in T$. Then since $\{a', b', y,x,  v_1, v_4\}$ does not induce a $P_2 + P_4$, by using \ref{F1-obs}  we may assume that $xa'\in E(G)$ and $xb'\notin E(G)$. But then $\{v_3,v_4,b',a',x,y\}$ induces a $P_2 + P_4$ which is a contradiction.  So the claim holds. $\sq$

\begin{claim}\label{F21-A12T}
  $G[A_1 \cup A_2 \cup T]$ is $K_3$-free.
\end{claim}	
{\it Proof of \cref{F21-A12T}}.~Suppose to the contrary that $G[A_1 \cup A_2 \cup T]$ contains a $K_3$  with vertices, say $x, y$ and $z$. By \cref{F21-Aij}, \cref{F21-K32} and \cref{lemmat}, we may assume that $x\in A_1$ and $y,z\in T$.
First suppose that $x = a$. Then by \ref{F1-obs}, we have $by,bz\notin E(G)$, and we may assume that $ba'\notin E(G)$, and then since $\{y, z, b, v_1, v_2, a'\}$ does not induce a $P_2 + P_4$, we may assume that $ya'\in E(G)$. But then since $\{x,y,a'\}$ does not induce a $K_3$ in   $G[A_1 \cup A_2 \cup T]$ which has  two vertices from $A_1 \cup A_2$, we have $xa'\notin E(G)$ (by \cref{F21-K32}), and then $\{v_3,v_4,b,x,y,a'\}$ induces a $P_2 + P_4$ which is a contradiction; so we may assume that $x\in A_1\setminus \{a,b\}$, by \ref{lemAi}. By \ref{F1-obs}, we may assume that $aa'\notin E(G)$. Then since $\{y,z,a,v_1,v_2,a'\}$ does not induce a $P_2+P_4$, we may assume that $ya\in E(G)$ or $ya'\in E(G)$. If $ya\in E(G)$, then $\{v_3,v_4,x,y,a,b\}$ induces a $P_2+P_4$ (by \ref{F1-obs}); so we may assume that $ya\notin E(G)$ and hence $ya'\in E(G)$. Then since $\{x,y,a'\}$ does not induce a $K_3$ in   $G[A_1 \cup A_2 \cup T]$ which has  two vertices from $A_1 \cup A_2$, we have $xa'\notin E(G)$, by \cref{F21-K32}. But then $\{v_3,v_4,x,y,z,a',b'\}$ induces a $P_2+P_4$ or a $K_4-e$ which is a contradiction. This proves \cref{F21-A12T}.   $\sq$

\begin{claim}\label{F21-chiA12T}
 $\chi(G[\{v_3,v_4\}\cup A_1 \cup A_2\cup T])= 2$.
\end{claim}
{\it Proof of \cref{F21-chiA12T}}.~Since $G[A_1 \cup A_2 \cup T]$ is ($P_4, K_3$)-free (by \cref{F21-Aij} and \cref{F21-A12T}), $G[A_1 \cup A_2 \cup T]$ is a perfect $K_3$-free graph and so   $\chi(G[A_1 \cup A_2\cup T])\leq 2$. Then since   $\{v_3,v_4\}$ is anticomplete to $A_1 \cup A_2\cup T$, we have $\chi(G[\{v_3,v_4\}\cup A_1 \cup A_2\cup T])= 2$.  This proves \cref{F21-chiA12T}. $\sq$

\medskip
\no Now from \cref{F21-Aij} and \cref{F21-chiA12T}, clearly   $\chi(G)\leq \chi(G[\{v_1,v_2\}\cup A_3 \cup A_4])+\chi(G[\{v_3,v_4\}\cup A_1 \cup A_2\cup T])=4$.
\end{proof}

\begin{theorem}\label{ai2}
	If $G$ contains an  $F$,	then $\chi(G) = 4$.
\end{theorem}
\begin{proof}
Suppose that $G$ contains an $F$. We label such an $F$ as shown in Figure~\ref{fig}.  Clearly  $a,b\in A_1$. We may assume that $G$ is ($H_1$, $H_2$)-free, by \Cref{ai3,ai22}. So $\omega(G[A_1]) = 2$ and $A_2, A_3$ and $A_4$ are stable sets. Then as in \cref{F21-Aij} of \cref{ai22}, we see that  $\chi(G[\{v_1,v_2\}\cup A_3\cup A_4])=2$  and
 $\chi(G[\{v_3,v_4\}\cup A_1\cup A_2])=2$. So if $T=\emptyset$, then $\chi(G)=4$ and we are done. We show that $T=\emptyset$ by using a sequence of claims given below.



\begin{claim2}\label{F2-ta-1}
If a vertex, say $t\in T$, is anticomplete to $\{a,b\}$, then it is anticomplete to $A\setminus \{a,b\}$.
\end{claim2}
{\it Proof of \cref{F2-ta-1}}.~For the sake of contradiction, suppose that $tx \in E(G)$ where $x\in A\setminus \{a,b\}$. If $x\in A_2$,  then $\{a, b, t, x, v_2, v_3\}$ or $\{v_3, v_4, t, x, a,b\}$ induces a $P_2 + P_4$ (by \ref{F1-obs}); so we may assume (by symmetry) that $x\in A_1\setminus \{a,b\}$ and that $t$ is anticomplete to $A_2\cup A_3\cup A_4$.  Then since  $\{t, x, a_2^*, v_2, v_3,v_4, a_3^*\}$   does not induce a $P_2 + P_4$, we may assume that $xa_2^* \in E(G)$, and then $\{a, b, t, x, a_2^*, v_2\}$ or $\{v_3, v_4,   x, a_2^*, a, b\}$ induces a $P_2+ P_4$ (by \ref{lemAi} and \ref{F1-obs}), which is a contradiction. So \cref{F2-ta-1} holds. $\sq$

	\begin{claim2} \label{F2-ta}
		Each vertex in $T$ has a neighbor in $\{a,b\}$.
	\end{claim2}
{\it Proof of \cref{F2-ta}}.~Recall that from \cref{lemmat}, we know that $T$ induces a $K_2$ or it is a stable set.
Suppose to the contrary that there is a vertex, say $t\in T$, such that $ta, tb \not\in E(G)$. Thus $t$ is anticomplete to $A$, by \cref{F2-ta-1}. Then since $G$ is connected, clearly  $G[T]\cong K_2$, and there is a vertex, say $t'\in T$, such that $tt'\in E(G)$. Then since $G$ is connected, $t'$ has a neighbor in $\{a,b\}$ (by \cref{F2-ta-1}), and then $\{v_3,v_4,a,b,t,t'\}$ induces a $P_2+P_4$ or $\{v_1,a,b,t'\}$ induces a $K_4-e$ which is a contradiction. So \cref{F2-ta} holds. $\sq$

\begin{claim2}\label{F2-ntind}
 $N(a) \cap T$  and  $N(b) \cap T$ are stable sets.
\end{claim2}
{\it Proof of \cref{F2-ntind}}.~Suppose to the contrary that there are adjacent vertices, say $x, y \in N(a) \cap T$. Then by \ref{F1-obs}, we have $xb, yb \notin E(G)$.  Since  $\{x, y, b, v_1, v_2, v_3, a_2^*\}$ does not induce a $P_2 + P_4$, $a_2^*$ has a neighbor in $\{x,y\}$ and thus  if $a_2^*$ has a nonneighbor in $\{a,x,y\}$, then $\{v_3,v_4, a_2^*,a,b,x,y\}$ induces a  $P_2+P_4$ or a $K_4-e$; so   $a_2^*$ is complete to $\{a,x,y\}$. Likewise,   $a_3^*$ is complete to $\{a,x,y\}$. But then $\{x,y,a,a_2^*,a_3^*\}$ induces a $K_5$ or a $K_4-e$ which is a contradiction. So \cref{F2-ntind} holds. $\sq$

\begin{claim2}\label{F2-twins}
$N(a) \cap T =\emptyset$. Likewise, $N(b) \cap T =\emptyset$.
\end{claim2}
{\it Proof of \cref{F2-twins}}.~Suppose to the contrary that there is a vertex, say $t \in N(a) \cap T$. Then $tb\notin E(G)$, by \cref{F2-ta}.   Then since $b$ and $t$ are not comparable vertices,
there are vertices, say $b'$ and $t'$, such that $tt', bb'\in E(G)$ and $tb', bt'\notin E(G)$. Clearly since $\omega(G[A_1])=2$, we have  $b'\notin A_1\setminus \{a,b\}$, by \ref{lemAi}.  Now if $t'\in (A_1\setminus \{a,b\})\cup T$, then $t'a\notin E(G)$ (by \ref{lemAi} and \cref{F2-ntind}) and then $\{v_3,v_4,b,a,t,t'\}$ induces a $P_2+P_4$; so we may assume (up to symmetry) that $t'\in A_2$. Then since $\{v_3, v_4, b,a, t,t'\}$ does not induce a $P_2 + P_4$, we have $at' \in E(G)$. Now if $b'\in A_2$, then since $A_2$ is a stable set, we have $b't'\notin E(G)$ and so by \ref{F1-obs}, $\{v_3, v_4, b',b,a,t'\}$ induces a $P_2 + P_4$; so we may assume (up to symmetry)  that $b'\in A_4$. Then by \ref{F1-obs}, $b'a\notin E(G)$ and then $\{v_2,v_3,t,a,b,b'\}$ induces a $P_2+P_4$ which is a contradiction.
 So  $N(a) \cap T =\emptyset$.   This proves  \cref{F2-twins}.  $\sq$

\smallskip

Thus by Claims \ref{F2-ta} and \ref{F2-twins}, we conclude  that $T = \emptyset$ and we are done.
\end{proof}

\subsection{$G$ is $F$-free}\label{sec:F2-free}

In this section, we show that if $G$ is $F$-free, then $\chi(G)=4$. First we observe that, since $G$ is $F$-free,  each $A_i$ is a stable set. Also, if $T=\emptyset$, then  we define the following stable sets: $S_1:= \{v_1\} \cup A_4$, $S_2:= \{v_2\} \cup A_1$, $S_3:= \{v_3\} \cup A_2$, and $S_4:= \{v_4\} \cup A_3$ so that $V(G) = S_1 \cup S_2 \cup S_3 \cup S_4$ and thus $\chi(G) = 4$; so we may assume that  $T\neq \es$. Next,  we prove some useful lemmas.

\begin{lemma}\label{lem-H}
 If $G$ is $F$-free, then $G$ is $H_3$-free.
\end{lemma}
\begin{proof}
Suppose to the contrary that $G$ contains an $H_3$. We label such an $H_3$  as shown in Figure~\ref{fig}.  Clearly  $a,b\in A_1$, $a',b'\in A_2$ and $t,t'\in T$. Recall that $a_3^*\in A_3$ and $a_4^*\in A_4$. Moreover, we have the following.
\begin{claim}\label{H-a34t}
 For $j\in \{3,4\}$, if $ta_j^*\in E(G)$,  then $a_j^*$ is anticomplete to $\{a,b,a',b'\}$.
\end{claim}
{\it Proof of \cref{H-a34t}}.~~We prove for $j=3$. Since $\{a_3^*,t,a,b,a',b'\}$ does not induce  a $K_4-e$ or an $F$, $a_3^*$ is anticomplete to $\{a,b,a',b'\}$. 
So \cref{H-a34t} holds. $\sq$

\begin{claim}\label{H-a34t-2}
 For $j\in \{3,4\}$: If  $ta_j^*\notin E(G)$ and if $a_j^*$ has a neighbor in $\{a,b\}$, then $a_j^*$   is complete to  $\{t',a,b\}$, and  is anticomplete to $\{a',b'\}$.  Likewise, if   $ta_j^*\notin E(G)$ and if $a_j^*$ has a neighbor in $\{a',b'\}$, then $a_j^*$   is complete to  $\{t',a',b'\}$, and is anticomplete to $\{a,b\}$.
\end{claim}
{\it Proof of \cref{H-a34t-2}}.~We prove for $j=3$. If $ta_3^*\notin E(G)$ and if $a_3^*$ has a neighbor in $\{a,b\}$, then since $\{v_2,v_4, a,t,b,a_3^*\}$ does not induce a $P_2+P_4$, $a_3^*$   is complete to  $\{a,b\}$, and then since $\{a,a',t,b,b', a_3^*\}$ does not induce a $K_4-e$, we see that $a_3^*$   is anticomplete to  $\{a',b'\}$.  Thus if $a_3^*t'\notin E(G)$, then $\{v_2,v_4, t', t, a, a_3^* \}$ induces a $P_2+P_4$.  So \cref{H-a34t-2} holds. $\sq$

Now, up to symmetry, we have three cases:

\vspace{-0.25cm}
 \begin{itemize}[leftmargin=0.5cm]\itemsep=0pt
\item If $ta_3^*, ta_4^*\in E(G)$, then by \cref{H-a34t}, $\{a_3^*,a_4^*\}$ is anticomplete to $\{a,b,a',b'\}$, and then $\{a,a',a_3^*,v_3,v_4,$ $ a_4^*\}$ or
$\{a_3^*,a_4^*, a,a',v_2, $ $b'\}$ induces a $P_2+P_4$ which is a contradiction.
\item If $ta_3^*, ta_4^*\notin E(G)$, then since $\{b', t, a_3^*, v_1,v_3, v_4, a_4^*\}$ does not induce a $P_2 + P_4$, we may assume $b'a_3^* \in E(G)$, and  hence $a_3^*$  is complete to  $\{t',a',b'\}$, and $a_3^*$ is anticomplete to $\{a,b\}$, by \cref{H-a34t-2}; but then $\{v_2, v_4, a_3^*, t', t, a\}$ induces a $P_2 + P_4$ which is a contradiction.

\item Finally, suppose that $ta_3^*\in E(G)$ and $ta_4^*\notin E(G)$. Then by \cref{H-a34t}, $a_3^*$ is anticomplete to $\{a,b,a',b'\}$. Also since $\{a,a',a_3^*, v_3, v_4,a_4^*\}$ or
$\{a_3^*, a_4^*,a',a,v_1, b\}$ does not induce a $P_2 + P_4$, we may assume that  $a_4^*$  is complete to  $\{t',a,b\}$, and  $a_4^*$ is anticomplete to $\{a',b'\}$, by \cref{H-a34t-2}. But then $\{v_1,v_3,a_4^*,t',t,b'\}$ induces a  $P_2 + P_4$ which is a contradiction.
 \end{itemize}
 The above contradiction show that $G$ is $H_3$-free. This proves \cref{lem-H}.
 \end{proof}

\begin{lemma}\label{tk2}
If $G$ is $F$-free and if $G[T]$ is isomorphic  to $K_2$, then $\chi(G) =4$.
\end{lemma}
\begin{proof}Let $T:=\{t,t'\}$ be such that $tt'\in E(G)$.   Now, we claim that $A_i=\{a_i^*\}$ for each $i$. Suppose to the contrary and up to symmetry that there are two vertices in $A_1$, say $a$ and $b$. Then since $a$ and $b$ are not comparable vertices,
there are vertices, say $a'$ and $b'$, such that $aa', bb'\in E(G)$ and $ab', ba'\notin E(G)$. If $a' \in T$ and $b' \in T\cup A_2$, then $\{v_3,v_4,a,a',b',b\}$ or $\{a,a',b,b',v_2,v_3\}$ induces a $P_2 + P_4$; so we may assume that $a',b'\in A_2$ or $a'\in A_2$ and $b'\in A_4$.  Also since  $G$ is $F$-free, $G$ is $H_3$-free, by \cref{lem-H}. First suppose that $a',b'\in A_2$. Then since $\{t,t',a,a',v_2,v_3\}$ does not induce a $P_2+P_4$, we may assume that $at\in E(G)$, and so $t$ is complete to $\{a,a',b,b'\}$, by \ref{lem-cru}. Then since $\{a, a', b, b',t,t'\}$ does not induce a $K_4-e$ or an $F$, clearly $t'$ is anticomplete to  $\{a,a',b,b'\}$. But then $C\cup \{a, a', b, b',t,t'\}$ induces an $H_3$ which is a contradiction. So, suppose that $a'\in A_2$ and  $b'\in A_4$. Then since $\{a,a',b,b',v_4,v_3\}$ does not induce a $P_2+P_4$, we have $a'b'\in E(G)$. Moreover we claim that:

\begin{claim}\label{tk2-case2}
The set $\{a,b\}$ is anticomplete to $T$.
\end{claim}
{\it Proof of \cref{tk2-case2}}.~If $a$ and $b$  have a common neighbor in $T$, say $t$, then since $\{v_3,v_4,a',a,t,b\}$ and  $\{v_2,v_3,b',b, $ $ t,a\}$ do not induce ($P_2+P_4$)'s, we have $ta',tb'\in E(G)$, and then $\{a,t,a',b'\}$ induces a $K_4-e$; so $a$ and $b$ do not have a common neighbor in $T$. Now if we assume (up to symmetry) that $bt\in E(G)$, then one of $\{v_3,v_4,b,t,a',a\}$ or $\{v_3,v_4,t,t',a',a\}$  or $\{v_3,v_4,b,t,t',a\}$ or  $\{t,t',a,a',v_2,v_3\}$  induces a $P_2+P_4$ which is a contradiction. So \cref{tk2-case2} holds. $\sq$

 \begin{claim}\label{tk2-case2-2}
The set $\{a',b'\}$ is complete to $T$.
\end{claim}
 {\it Proof of \cref{tk2-case2-2}}.~Since $\{t,t',b,v_1,a,a'\}$ does not induce a $P_2+P_4$, we see that $a'$ has a neighbor in $\{t,t'\}$ (by \cref{tk2-case2}), and then since $\{v_3,v_4,a,a',t,t'\}$ does not induce a $P_2+P_4$, $a'$ is complete to $\{t,t'\}$.  Likewise, $b'$ is complete to $\{t,t'\}$. This proves \cref{tk2-case2-2}. $\sq$

Now since $\{t,t',a,v_1,v_2,v_3, a_3^*\}$ does not induce a $P_2+P_4$, clearly $a_3^*$ has a neighbor in $\{t,t'\}$, say $t$. Also since $\{a',b',t,t'\}$ is a clique (by \cref{tk2-case2-2}), we conclude that $N(a_3^*)\cap \{a',b',t,t'\} =\{t\}$ or $N(a_3^*)\cap \{a',b',t,t'\} =\{a',b',t,t'\}$, by \cref{obs1}$(i)$.  But then $\{a',t',b,v_1,v_3,v_4, a_3^*\}$ induces a $P_2+P_4$ or $\{a_3^*,a',b',t,t'\}$ induces a $K_5$ which is a contradiction.

\medskip
So we conclude that $G[A_i] \simeq K_1$, and hence $V(G) = \{v_1, v_2, v_3, v_4, a_1^*, a_2^*, a_3^*,a_4^*, t, t'\}$. So $G$ is a ($P_2+P_4,  K_4-e, K_5, F$)-free graph with $10$ vertices. Now it is not hard to verify that $\chi(G)=4$. 
\end{proof}

\setcounter{claim2}{0}

\begin{theorem}\label{lemma-F2-free}
If $G$ is $F$-free, then $\chi(G)= 4$.
\end{theorem}
\begin{proof}
From \Cref{lemmat,tk2}, we may assume that  $T$ is a stable set.
Next, we show the following.

\begin{claim2}\label{tk1-T}
$T$ is a singleton set.
\end{claim2}
{\it Proof of \cref{tk1-T}}.~Suppose to the contrary  that there are two vertices in $T$, say $t$ and $t'$. Then since $t$ and $t'$ are not comparable vertices,
there are vertices, say $a$ and $b$, such that $ta, t'b\in E(G)$ and $tb, t'a\notin E(G)$. Since $T$ is a stable set, we may assume that $a,b\in A$. Now if $a\in A_1$ and $b\in A_2$, then  $\{t,a,t',b,v_2,v_3\}$ or $\{v_3,v_4,t,a,b,t'\}$ induces a $P_2 + P_4$; so we may assume that $a,b\in A_1$. Then since $\{t,a,a_2^*,v_2,v_3,v_4,a_3^*\}$  does not induce  a $P_2+P_4$, we may assume that $a_2^*$ has a neighbor in $\{a,t\}$, and then since $\{t',b,a,t,a_2^*,v_2,v_3\}$ or $\{v_3,v_4,a_2^*,t,a,t',b\}$ does not induce  a $P_2 + P_4$, we conclude that $a_2^*$ is complete to
   $\{t,a,t',b\}$. Likewise, by using similar arguments, we may assume that $a_3^*$ has a neighbor in $\{a,t\}$, and hence $a_3^*$ is also complete to $\{t,a,t',b\}$. But then $\{a_2^*,a_3^*,t,a,t'\}$ induces a $K_4-e$ which is a contradiction. So $T$ is a singleton set. This proves \cref{tk1-T}. $\sq$

\medskip

By \cref{tk1-T}, we let $T = \{t\}$. Now if $t$ is anticomplete to $A_i$, for some $i$, say $i = 1$, then we define the following stable sets: $S_1:= \{v_1\} \cup A_4$, $S_2:= \{v_2, t\} \cup A_1$, $S_3:= \{v_3\} \cup A_2$, and $S_4:= \{v_4\} \cup A_3$ so that $V(G) = S_1 \cup S_2 \cup S_3 \cup S_4$ and thus $\chi(G) = 4$; so we may assume that  $t$ has a neighbor in $A_i$, for every $i$; and let  $ta_i^*\in E(G)$.

\begin{claim2}\label{tk1-Ai}
For each $i$,	we have $A_i =\{a_i^*\}$.
\end{claim2}
{\it Proof of \cref{tk1-Ai}}.~We prove for $i=1$. Suppose to the contrary that there are vertices in $A_1$, say $a~(=a_1^*)$ and $b$. Since $a$ and $b$ are not comparable vertices, there are vertices, say $a'$ and $b'$, such that $aa', bb'\in E(G)$ and $ab', ba'\notin E(G)$. Now if (up to symmetry)   $a'=t$ and $b'\in A_2$, then since $\{a,a', b,b',v_2,v_3\}$ does not induce a $P_2 + P_4$, we have $a'b'\in E(G)$ and then $\{v_3,v_4,a,a', b',b\}$  induces a $P_2 + P_4$; so we may assume that $a',b'\in A_2$ or $a'\in A_2$ and $b'\in A_4$. Now if $a'\in A_2$ and  $b'\in A_4$, then as in the proof of \cref{tk2-case2}, we see that $bt\notin E(G)$, and then  $\{a,t, b,b', v_4,v_3\}$ or $\{v_2,v_3, b,b', t,a\}$ induces a $P_2+P_4$ which is a contradiction; so suppose that $a',b'\in A_2$. Then since $at\in E(G)$, clearly $t$ is complete to    $\{a,a',b,b'\}$, by \ref{lem-cru}. Moreover, since  $ta_3^*, ta_4^*\in E(G)$, as in \cref{H-a34t},   we conclude that  $\{a_3^*,a_4^*\}$ is anticomplete to $\{a,b,a',b'\}$. But then $\{a,a',a_3^*, v_3,v_4,a_4^*\}$ or $\{a_3^*,a_4^*,a,a',v_2,b'\}$ induces a $P_2 + P_4$ which is a contradiction. So \cref{tk1-Ai} holds. $\sq$

\medskip
 Thus from above claims, we observe that $V(G) = \{v_1, v_2, v_3, v_4, a_1^*, a_2^*, a_3^*,a_4^*, t\}$. So $G$ is a ($P_2+P_4, K_4-e, K_5, F$)-free graph with $9$ vertices. Now it is not hard to verify that $\chi(G)=4$.
\end{proof}

\medskip
      \no{\bf Proof of \cref{mainthm-2}}.
 This follows from \cref{ai2,lemma-F2-free}. \hfill{$\Box$}

 \bigskip
\noindent{\bf Acknowledgement}.
The first author acknowledges the National Board of Higher Mathematics (NBHM), Department of Atomic Energy, India for the financial support (No.  02011/55/2023/R\&D-II/3736) to carry out this research work, and CHRIST (Deemed to be University), Bengaluru for permitting to pursue the NBHM project.

\end{document}